\documentclass{article}
\usepackage[utf8]{inputenc}
\usepackage{amsmath,amssymb,amsthm}
\usepackage{bbm}
\usepackage{xcolor}
\usepackage{verbatim}
\usepackage{caption}
\usepackage{subcaption}
\usepackage{graphicx}
\usepackage{authblk}
\usepackage[colorlinks=true,linkcolor=blue]{hyperref}

\usepackage[left=2.5cm,right=2.5cm,top=2.5cm,bottom=2.5cm]{geometry}

\newtheorem{theorem}{Theorem}

\newtheorem{lemma}[theorem]{Lemma}

\newtheorem{remark}[theorem]{Remark}

\newcommand{\A}{\mathbf{A}}
\newcommand{\B}{\mathbf{B}}
\newcommand{\C}{\mathbf{C}}

\newcommand{\e}{\mathbf{e}}
\newcommand{\M}{\mathbf{M}}

\renewcommand{\u}{\mathbf{u}}

\renewcommand{\v}{\mathbf{v}}

\newcommand{\w}{\mathbf{w}}
\newcommand{\x}{\mathbf{x}}
\newcommand{\X}{\mathbf{X}}

\newcommand{\Y}{\mathbf{Y}}
\newcommand{\z}{\mathbf{z}}
\newcommand{\Z}{\mathbf{Z}}

\newcommand{\id}{\mathbf{Id}}
\newcommand{\LAmbda}{\boldsymbol{\lambda}}
\newcommand{\LAMBDA}{\boldsymbol{\Lambda}}

\newcommand{\SIGMA}{\boldsymbol{\Sigma}}
\newcommand{\SIGMADITH}{\boldsymbol{\Sigma}^{\operatorname{dith}}}
\newcommand{\SIGMAADAP}{\boldsymbol{\Sigma}^{\operatorname{adap}}}

\newcommand{\Tau}{\boldsymbol{\tau}}

\newcommand{\XI}{\boldsymbol{\Xi}}
\newcommand{\THETA}{\boldsymbol{\Theta}}

\newcommand{\0}{\boldsymbol{0}}

\newcommand{\R}{\mathbb{R}}
\newcommand{\W}{\boldsymbol{W}}

\newcommand{\E}{\mathbb{E}}

\renewcommand{\P}[2][]{\;\mathbb{P}_{#1}\!\left[#2\right]}
\newcommand{\pnorm}[2]{\left\| #1 \right\|_{#2}}
\newcommand{\tnorm}[1]{{\left\vert\kern-0.25ex\left\vert\kern-0.25ex\left\vert #1 
    \right\vert\kern-0.25ex\right\vert\kern-0.25ex\right\vert}}
\newcommand{\abs}[2]{\left| #1 \right|_{#2}}

\newcommand{\round}[1]{\left( #1 \right)}

\newcommand{\curly}[1]{\left\{ #1 \right\} }

\newcommand{\diag}{\mathrm{diag}}

\newcommand{\sign}{\mathrm{sign}}

\newcommand{\tr}{\mathrm{Tr}}

\newcommand{\bP}{\mathbb{P}}
\newcommand{\Om}{\Omega}
\newcommand{\cF}{\mathcal{F}}
\newcommand{\cG}{\mathcal{G}}

\definecolor{tumor}{HTML}{E37222}
\definecolor{tumred}{RGB}{205,32,44}
\definecolor{tumblue}{RGB}{00,115,207}
\definecolor{tumgreen}{RGB}{162,173,000}

\newcommand{\sjoerd}[1]{{{#1}}}

\newcommand{\sjoerdnewred}[1]{\textcolor{red}{#1}}
\newcommand{\sjoerdgreen}[1]{{{#1}}}
\newcommand{\johannes}[1]{{#1}}
\newcommand{\johannesnew}[1]{{#1}}
\newcommand{\newJ}[1]{#1}
\newcommand{\newS}[1]{#1}

\newcommand{\blue}[1]{{ #1}}
\newcommand{\sjoerdredred}[1]{{ #1}}

\newcommand{\noteJ}[1]{{\textcolor{teal}{\textbf{JM:} #1 }}}

\allowdisplaybreaks

\title{Tuning-free one-bit covariance estimation \\ using data-driven dithering}
\author[$\star$]{Sjoerd Dirksen}
\author[$\dagger$,$\ddagger$]{Johannes Maly}
\affil[$\star$]{\normalsize Mathematical Institute, Utrecht University, The Netherlands}
\affil[$\dagger$]{Department of Mathematics, LMU Munich, Germany}
\affil[$\ddagger$]{Munich Center for Machine Learning (MCML)}
\date{}

\begin{document}

\maketitle

\begin{abstract}
\noindent We consider covariance estimation of any subgaussian distribution from finitely many i.i.d.\ samples that are quantized to one bit of information per entry. Recent work has shown that a reliable estimator can be constructed if uniformly distributed dithers on $[-\lambda,\lambda]$ are used in the one-bit quantizer. This estimator enjoys near-minimax optimal, non-asymptotic error estimates in the operator and Frobenius norms if $\lambda$ is chosen proportional to the largest variance of the distribution. However, this quantity is not known a-priori, and in practice $\lambda$ needs to be carefully tuned to achieve good performance. In this work we resolve this problem by introducing a tuning-free variant of this estimator, which replaces $\lambda$ by a data-driven quantity. We prove that this estimator satisfies the same non-asymptotic error estimates --- up to small (logarithmic) losses and a slightly worse probability estimate. \sjoerdredred{We also show that by using refined data-driven dithers that vary per entry of each sample, one can construct an estimator satisfying the same estimation error bound as the sample covariance of the samples before quantization --- again up logarithmic losses. Our proofs rely on a new version of the Burkholder-Rosenthal inequalities for matrix martingales, which is expected to be of independent interest.}   
\end{abstract}
  
%

\section{Introduction}
\label{sec:Introduction}

A classical problem in multivariate statistics is to estimate covariance matrices of high-dimensional probability distributions from a finite number of samples. Application areas in which covariance estimation is of relevance include financial mathematics \cite{ledoit2003improved}, pattern recognition \cite{dahmen2000structured}, and signal processing and transmission \sjoerd{\cite{haghighatshoar2018low,krim1996two}}. In the past decades, a whole line of research has been devoted to quantifying how many i.i.d.\ samples $\X_1,...,\X_n \overset{\mathrm{d}}{\sim} \X$ of a mean-zero\footnote{Since estimating and subtracting the mean of the distribution can be considered as a data pre-processing step, most works assume that $\X$ has zero mean for simplicity.} random vector $\X \in \R^p$ suffice to reliably approximate the covariance matrix $\SIGMA= \E(\X\X^T) \in \R^{p\times p}$ up to a given accuracy. It is widely known that the covariance matrix of well-behaved distributions, such as Gaussian distributions, can be estimated via the sample covariance matrix with a near-optimal sampling rate $n \sim p$ \cite{cai2010optimal}. In the last two decades, more sophisticated estimators have been devised to obtain similar results for heavy-tailed distributions (see, e.g., \cite{ke2019user,mendelson2018robust} and the references therein) 
or to further reduce the sample complexity by leveraging structures of $\SIGMA$ such as a sparse \cite{bickel2008covariance,chen2012masked,levina2012partial}, low rank \cite{ke2019user}, or Toeplitz structure \cite{cai2013optimal,kabanava2017masked,Maly2022}. 

In our recent work \cite{dirksen2021covariance}, we considered a very different challenge in covariance estimation, namely \textit{coarse quantization} of the samples, i.e., each entry of the vectors $\X_k$ is mapped to one or two bits of information by a chosen \textit{quantizer}. Coarse quantization occurs naturally in various applications. In signal processing, analog samples are collected via sensors and hence need to be quantized to finitely many bits before they can be digitally transmitted. The number of quantization bits determines the energy efficiency of a sensor to a significant degree. If the number of sensors is large, fine quantization can be prohibitive to use. Even in purely digital systems, coarse quantization can be attractive to reduce storage and processing costs, e.g., in the context of neural networks, see for instance the surveys \cite{guo2018survey,  deng2020model, gholami2021survey} and the recent rigorous mathematical studies \cite{lybrand2021greedy,maly2022simple}. Our current work is motivated by covariance estimation problems arising in signal processing, especially with large antenna arrays \cite{jacovitti1994estimation,roth2015covariance,bar2002doa,choi2016near,li2017channel}. We discuss this direction in more detail in Section~\ref{sec:RelatedWork}. 

In this paper we focus on memoryless one-bit quantizers of the form $Q:\R^p\to\{-1,1\}^p$, $Q(\X) = \sign(\X+\Tau)$, where $\sign$ is the sign function acting element-wise on its input and $\Tau\in \R^p$ represents a vector of quantization thresholds.\footnote{This is an example of a \emph{memoryless scalar quantizer}, which means that vectors are quantized entry-wise and vector entries are quantized independently of each other. This type of quantization is very popular in the literature, as it allows for a simple, energy efficient hardware implementation. For alternative, more elaborate vector quantization schemes we refer to, e.g., \cite{gersho2012vector,gray1987oversampled,tewksbury1978oversampled}.} If $\Tau$ is random, then this method is called one-bit quantization with \emph{dithering} in engineering literature. Our starting point is the work \cite{dirksen2021covariance}, which proposed a setting in which each sample $\X_k$ is quantized entry-wise to two bits
\begin{align} \label{eq:TwoBitSamples}
    \sign(\X_k + \Tau_k) \quad \text{ and } \quad \sign(\X_k + \bar{\Tau}_k),
\end{align}
where the dithering vectors $\Tau_1,\bar{\Tau}_1,\ldots,\Tau_n,\bar{\Tau}_n$ are independent and uniformly distributed in $[-\lambda,\lambda]^p$, for fixed $\lambda>0$. The authors of \cite{dirksen2021covariance} proposed to recover $\SIGMA$ by the estimator
\begin{align} \label{eq:TwoBitEstimator}
\SIGMADITH_n = \frac{1}{2}\johannesnew{\tilde{\SIGMA}_n} + \frac{1}{2}(\johannesnew{\tilde{\SIGMA}_n})^T,
\end{align}
where
\begin{align} \label{eq:AsymmetricEstimator}
\johannesnew{\tilde{\SIGMA}_n} = \frac{\lambda^2}{n}\sum_{k=1}^n \sign(\X_k + \Tau_k)\sign(\X_k + \bar{\Tau}_k)^T
\end{align}
is a re-scaled and asymmetric version of the sample covariance matrix of the samples in \eqref{eq:TwoBitSamples}. To formulate the estimation error estimate from \cite{dirksen2021covariance}, we denote for $\Z \in \R^{p\times p}$ its operator norm by $\pnorm{\Z}{}~=~\sup_{\u \in \mathbb{S}^{p-1}} \pnorm{\Z\u}{2}$, its entry-wise max-norm by $\pnorm{\Z}{\infty} = \max_{i,j} \abs{Z_{i,j}}{}$, and its maximum column norm \sjoerd{by} $\pnorm{\Z}{1\rightarrow 2} = \max_{j \in [p]} \pnorm{\z_j}{2}$, where $\z_j$ denotes the $j$-th column of $\Z$. We use $\lesssim_K$ to write an inequality that hides a multiplicative positive constant that only depends on $K$ and write $a \simeq_K b$ if $a \lesssim_K b \lesssim_K a$.

\begin{theorem}[{\cite[Theorem 4]{dirksen2021covariance}}] \label{thm:OperatorDitheredMask_Old}
Let $\X$ be a mean-zero, $K$-subgaussian\footnote{The formal definition of a subgaussian random vector is provided in Section \ref{sec:Notation} below.} vector with covariance matrix \sjoerd{$\SIGMA$}. Let $\X_1,...,\X_n \overset{\mathrm{\sjoerd{i.i.d.}}}{\sim} \X$. Let $\M \in [0,1]^{p\times p}$ be a fixed symmetric mask. If $\lambda^2 \gtrsim_{\sjoerdgreen{K}} \log(n) \|\SIGMA\|_{\infty}$, then with probability at least $1-e^{-t}$,  
	$$\pnorm{\M \odot \SIGMADITH_n - \M \odot \SIGMA}{}\lesssim_{\sjoerdgreen{K}}
	\|\M\|_{1\to 2}(\lambda\|\SIGMA\|^{1/2}+\lambda^2)\sqrt{\frac{\log(p)+t}{n}} + \lambda^2\|\M\| \frac{\log(p)+t}{n},$$
    where $\odot$ denotes the Hadamard (i.e., entry-wise) product.
	In particular, if $\lambda^2 \simeq_{\sjoerdgreen{K}} \log(n) \|\SIGMA\|_{\infty}$, then
	\begin{align}
	\label{eqn:OperatorDitheredMask}
	\begin{split}
	    &\pnorm{\M \odot  \SIGMADITH_n - \M \odot \SIGMA}{}  \lesssim_{\sjoerdgreen{K}}
	    \log(n) \|\M\|_{1\to 2}\sqrt{\frac{\|\SIGMA\| \ \|\SIGMA\|_{\infty}(\log(p)+t)}{n}} + \log(n)\|\M\|\|\SIGMA\|_{\infty}\frac{\log(p)+t}{n}. 
	\end{split}
	\end{align}	
%
%
%
\end{theorem}

Note that this result incorporates sparsity priors on $\SIGMA$ in the form of the mask $\M$ and is of the same shape as state-of-the-art performance estimates for masked estimators in the `unquantized' setting \cite{chen2012masked}. If $\SIGMA$ has no structure, then one can take a trivial all-ones mask (in which case $\|\M\|_{1\to 2}=\sqrt{p}$ and $\|\M\|=p$) to recover the minimax optimal rate up to log-factors. Analogous results for Frobenius- and entry-wise max-norm estimation errors can be found in \cite{yang2023plug}. \blue{As was already discussed in \cite{dirksen2021covariance}, the main weakness of Theorem~\ref{thm:OperatorDitheredMask_Old} lies in the required oracle knowledge of $\SIGMA$ for tuning the hyperparameter $\lambda$. Indeed, any choice of $\lambda$ that is too small or too large will either break the guarantees or deteriorate the error bound, and this reflects the actual behavior of \eqref{eq:TwoBitEstimator}: as $\lambda$ increases a U-shaped performance curve can consistently be observed in numerical experiments \cite{dirksen2021covariance,yang2023plug}. The goal of this work is to resolve this issue.}

\subsection{Notation}
\label{sec:Notation}

\blue{
Let us clarify the notation that will be used throughout this work before continuing. We write $[n] = \{ 1,...,n \}$ for $n\in \mathbb{N}$. \sjoerd{We use the notation $a \lesssim_{\alpha} b$ (resp.\ $\gtrsim_{\alpha}$) to abbreviate $a \le C_{\alpha}b$ (resp.\ $\ge$), for a constant $C_{\alpha} > 0$ depending only on $\alpha$. Similarly, we write $a \lesssim b$ if $a \le Cb$ for an absolute constant $C>0$.} \sjoerdgreen{We write $a\simeq b$ if both $a\lesssim b$ and $b\lesssim a$ hold (with possibly different implicit constants)}. The values of absolute constants $c,C > 0$ may vary from line to line. We let scalar-valued functions act entry-wise on vectors and matrices. In particular,
\begin{align*}{}
[\sign(\x)]_i = \begin{cases}
1 & \text{if } x_i\geq 0 \\
-1& \text{if } x_i<0,
\end{cases}
\end{align*}
for all $\x\in \R^p$ and $i\in [p]$. A mean-zero random vector $\X$  \sjoerd{in $\R^p$} is called $K$-subgaussian if $$\|\langle \X,\x\rangle\|_{\psi_2} \leq K \johannes{(\mathbb{E}\langle \X,\x\rangle^2)^{1/2}} \quad \mbox{ for all }  \x \in \sjoerd{\R^p},$$
where the subgaussian ($\psi_2$-)norm of a random variable $X$ is defined by 
\begin{align*}
    \pnorm{X}{\psi_2} = \inf \curly{ t>0 \colon \johannes{\E\round{ \exp\round{\frac{X^2}{t^2}} }} \le 2 }.
\end{align*}
For $1\leq q<\infty$, the $L^q$-norm of a random variable $X$ is denoted by 
\begin{align*}
    \| X \|_{L^q} = ( \E X^q )^{\frac{1}{q}}.
\end{align*}
For $\Z \in \R^{\sjoerdredred{p_1\times p_2}}$, we denote the operator norm by $\pnorm{\Z}{}~=~\sup_{\u \in \mathbb{S}^{p_2-1}} \pnorm{\Z\u}{2}$, the entry-wise max-norm by $\pnorm{\Z}{\infty} = \max_{i,j} \abs{Z_{i,j}}{}$, the Frobenius norm by $\| \Z \|_F~=~(\sum_{i,j} |Z_{i,j}|^2)^{1/2}$, the maximum column norm \sjoerd{by} $\pnorm{\Z}{1\rightarrow 2} = \max_{j \in [p_2]} \pnorm{\z_j}{2}$, where $\z_j$ denotes the $j$-th column of $\Z$\sjoerdredred{ , and, for $1\leq q<\infty$, the Schatten $q$-norm by $\|Z\|_{S^q} = (\tr[(Z^TZ)^{q/2}])^{1/q}$.} The Hadamard (i.e., entry-wise) product of two matrices is denoted by $\odot$. \newJ{We abbreviate Hadamard powers by $\Z^{\odot k} = \Z \odot \cdots \odot \Z$.} We denote the \sjoerd{all ones}-matrix by $\boldsymbol{1} \in \R^{p\times p}$.
The $\diag$-operator, when applied to \sjoerd{a} matrix, extracts the diagonal as \sjoerd{a} vector; when applied to a vector, it outputs the corresponding diagonal matrix. 
For symmetric \johannes{$\W,\Z\in \R^{p\times p}$ we write $\W \preceq \Z$ if $\Z-\W$} is positive semidefinite. We will use that for any conditional expectation 
$\E_{\mathcal{G}}=\E(\cdot|\mathcal{G})$,
	\begin{align} \label{eq:Kadison}
	\johannes{\mathbb{E}_{\mathcal{G}}\Z^T\mathbb{E}_{\mathcal{G}}\Z \preceq \mathbb{E}_{\mathcal{G}}(\Z^T\Z) \quad \text{for any } \Z\in \R^{p_1\times p_2}}{},
	\end{align}
	which is immediate from 
	$$\0 \preceq \mathbb{E}_{\mathcal{G}}[(\Z-\mathbb{E}_{\mathcal{G}}\Z)^T(\Z-\mathbb{E}_{\mathcal{G}}\Z)].$$ 
}

\blue{
\subsection{Organization}

We present our main results in Section~\ref{sec:MainResults}. Section~\ref{sec:Numerics} validates our theoretical results in numerical simulations. \sjoerdredred{We compare our results with related work in Section~\ref{sec:RelatedWork}.} All proofs can be found in Section \ref{sec:Proofs} and Appendix \ref{sec:ProofAppendix}.
}

\section{Main results}
\label{sec:MainResults}

\sjoerdredred{We will present results of two different flavours. First, we introduce a one-bit quantizer with \emph{data-adaptive dithering}: instead of using uniform dithers with a fixed parameter $\lambda$ that scales in terms of the (unknown) maximal variance $\|\SIGMA\|_{\infty}$, we will use data-driven dithering parameters that estimate (a scaled version of) $\|\SIGMA\|_{\infty}$. We then solve the main problem posed in this work by showing that an estimator based on these quantized samples satisfies essentially the same error estimate as in Theorem~\ref{thm:OperatorDitheredMask_Old}, see Theorem~\ref{thm:OperatorDitheredMask}. Second, we consider a refined scenario where we again use data-adaptive dithers, but allow the dithering scales to differ across the entries of each sample. This refined approach, inspired by \cite{chen2023parameter}, yields an estimator that satisfies a superior estimation error estimate, presented in Theorem~\ref{thm:OperatorDitheredMaskDecoupled}, but at the same time has major practical drawbacks, pointed out in Remark~\ref{rem:Decoupled}.}  

%
%
%
%
%
%
%

\subsection{Data adaptive dithering}

We assume\footnote{Note that we added an additional sample $\X_0$ for convenience, which is only used to provide an initial estimate of the adaptive dither.} that $\X_0,...,\X_n \overset{\mathrm{i.i.d.}}{\sim} \X$ and again consider quantized samples of the form
\begin{align} \label{eq:TwoBitSamplesAdaptive}
    \sign(\X_k + \Tau_k) \quad \text{ and } \quad \sign(\X_k + \bar{\Tau}_k),
\end{align}
for $1 \le k \le n$. However, now the dithering vectors $\Tau_k,\bar{\Tau}_k$ are independent and uniformly distributed in $[-\tilde\lambda_k,\tilde\lambda_k]^p$, where $\tilde\lambda_k = c \log(k) \cdot \lambda_k$, for $c > 0$ an absolute constant (to be specified later) and $\lambda_k$ defined by 
\begin{align} \label{eq:lambdaAdaptive}
    \lambda_{k}
    = \frac{1}{k} \sum_{j=0}^{k-1} \| \X_j \|_\infty
    = \frac{1}{k} \left( (k-1) \lambda_{k-1} + \| \X_{k-1} \|_\infty \right), \qquad 1\leq k\leq n.
\end{align}
%
Note that each $\lambda_k$ is an estimator of the maximal variance $\| \SIGMA \|_\infty$ of $\X$ based on the first $k$ samples. In particular, the dithers $\Tau_k,\bar{\Tau}_k$ used to quantize the $k$-th sample $\X_k$ only depend on the previous samples via $\tilde\lambda_k$ but are independent of $\X_k$ and all future samples \sjoerdredred{(in the language of stochastic processes, the sequence of dithers is \emph{predictable})}. Since $\lambda_{k}$ can be computed in real time from the previous iterate $\lambda_{k-1}$ and the sample $\X_{k-1}$ before quantization, cf.\ \eqref{eq:lambdaAdaptive}, the dithering parameters can be efficiently updated in hardware implementation. 

Based on the samples in \eqref{eq:TwoBitSamplesAdaptive}, we now define a covariance estimator by
\begin{align} \label{eq:TwoBitEstimatorAdaptive}
\SIGMAADAP_n = \frac{1}{2}\SIGMA'_n + \frac{1}{2}(\SIGMA'_n)^T,
\end{align}
where
\begin{align} \label{eq:AsymmetricEstimatorAdaptive}
\SIGMA'_n = \sum_{k=1}^n \frac{\tilde\lambda_k^2}{n} \sign(\X_k + \Tau_k)\sign(\X_k + \bar{\Tau}_k)^T.
\end{align}

\begin{remark}
\blue{The adaptive dithering introduced in \eqref{eq:lambdaAdaptive} comes with two obvious drawbacks:}
\begin{enumerate}
    \item In addition to storing the two-bit samples in \eqref{eq:TwoBitSamplesAdaptive} one also needs to keep track of the sequence $\tilde\lambda_k$ which has to be stored with higher precision, e.g., in standard $32$-bit representation. Hence, the proposed quantization scheme requires $(2p + 32)$ bits per sample vector. Compared with using high precision samples which would require $32$ bits per entry and thus $32p$ bits per vector, this still results in a massive reduction of storage and computing time if the ambient dimension $p$ is of moderate size.
    \item \blue{\sjoerdredred{Implementing an update rule like \eqref{eq:lambdaAdaptive} in hardware requires a feedback system, as is also used in more sophisticated noise shaping quantizers \cite{gray1987oversampled}.} 
    }
\end{enumerate}
\end{remark}

Our first result bounds the estimation error of $\SIGMAADAP_n$ in the max and Frobenius norms.

\begin{theorem}
\label{thm:FrobeniusDitheredMask}
For any $K>0$, there exist constants $C_1,C_2>0$ depending only on $K$ such that the following holds. Let $\X \in \mathbb R^p$ be a mean-zero, $K$-subgaussian vector with covariance matrix \sjoerd{$\SIGMA \in \R^{p\times p}$}.
Let $\X_0,...,\X_n \overset{\mathrm{i.i.d.}}{\sim} \X$ and let $\M \in [0,1]^{p\times p}$ be a fixed symmetric mask. If the adaptive dithering parameters satisfy $\tilde\lambda_k^2 = C_1 \log(k) \lambda_k^2$, where $\lambda_k$ is defined in \eqref{eq:lambdaAdaptive}, then for any $\theta \ge C_2\max\{\log(p),\log(n)\}$ we have with probability at least $1 - 2e^{-\theta}$   
    \begin{align*}
        \pnorm{\M \odot \SIGMAADAP_n - \M \odot \SIGMA}{\infty}
        \lesssim_{K} \| \M \|_\infty \| \SIGMA \|_\infty\log(p) \log(n)\left(\frac{\theta^{3/2}}{\sqrt{n}} + \frac{\theta^2}{n}\right)
    \end{align*}
    and 
    \begin{align*}
        \pnorm{\M \odot \SIGMAADAP_n - \M \odot \SIGMA}{F}
        \lesssim_{K} \| \M \|_F \| \SIGMA \|_\infty\log(p) \log(n)\left(\frac{\theta^{3/2}}{\sqrt{n}} + \frac{\theta^2}{n}\right).
    \end{align*}
\end{theorem}


Our second result, which is the main result of this work, estimates the operator norm error.

\begin{theorem} \label{thm:OperatorDitheredMask}
For any $K > 0$, there exist constants $C_1,C_2>0$ depending only on $K$ such that the following holds. Let $\X \in \R^p$ be a mean-zero, $K$-subgaussian vector with covariance matrix \sjoerd{$\SIGMA \in \R^{p\times p}$}.
Let $\X_0,...,\X_n \overset{\mathrm{i.i.d.}}{\sim} \X$ and let $\M \in [0,1]^{p\times p}$ be a fixed symmetric mask. If $\tilde\lambda_k^2 = C_1 \log(k) \lambda_k^2$, where $\lambda_k$ is defined in \eqref{eq:lambdaAdaptive}, then for any $\theta \ge C_2\max\{\log(p),\log(n)\}$ we have with probability at least $1 - 2e^{-\theta}$   
	$$\pnorm{\M \odot \SIGMAADAP_n - \M \odot \SIGMA}{}\lesssim_{K} 
	\log(p)\log(n)\left(\theta^{5/2}\frac{\|\M\|_{1\to 2}}{\sqrt{n}}\|\SIGMA\|_{\infty}^{1/2}\|\SIGMA\|^{1/2}+\theta^3\frac{\|\M\|}{n} \|\SIGMA\|_{\infty}\right).
	$$
\end{theorem}

\sjoerdredred{If $\M$ is the all-ones matrix, then $\| \M \| = \| \M \|_F = p$ and $\| \M \|_{1\to 2} = \sqrt{p}$ so that the bounds in Theorems \ref{thm:FrobeniusDitheredMask} and \ref{thm:OperatorDitheredMask} recover the minimax optimal rates in \cite{cai2010optimal} (up to logarithmic factors). If $\SIGMA$ has a known (approximate) sparsity pattern, in the sense that we can find an (approximately) sparse $\M$ such that $\M\odot \SIGMA - \SIGMA$ has small norm, then Theorems \ref{thm:FrobeniusDitheredMask} and \ref{thm:OperatorDitheredMask} can yield better error rates for estimating $\SIGMA$. For instance, in certain applications it is known that the entries of $\SIGMA$ decay away from the diagonal, so that banding and tapering masks are effective.} \blue{For a further discussion of masked covariance estimation we refer the reader to \cite{chen2012masked,levina2012partial} \sjoerdredred{and the references therein.}}

Remarkably, up to small losses in terms of logarithmic factors and a slightly worse (but still exponentially decaying) failure probability, the error bounds in Theorems~\ref{thm:FrobeniusDitheredMask} and \ref{thm:OperatorDitheredMask} match the ones that were previously obtained for the estimator $\SIGMADITH_n$ -- which required oracle knowledge of the maximal variance $\|\SIGMA\|_{\infty}$ of the distribution, c.f.\ \cite[Theorem 1]{yang2023plug} and Theorem~\ref{thm:OperatorDitheredMask_Old}, respectively.

\begin{remark}
\blue{Let us highlight two important points regarding the constant $C_1$ appearing in Theorems \ref{thm:FrobeniusDitheredMask} and \ref{thm:OperatorDitheredMask}:}
\begin{enumerate}
    \item In both theorems it is sufficient to pick $\tilde\lambda_k^2 \geq C_1 \log(k) \lambda_k^2$, i.e., in practice the scaling $C_1$ does not need to be precisely known. Indeed, setting $\tilde\lambda_k^2 = c \cdot C_1 \log(k) \lambda_k^2$ for $c \ge 1$ only results in an additional scaling of the error estimates by $c$. The constant $C_1$ in Theorems \ref{thm:FrobeniusDitheredMask} and \ref{thm:OperatorDitheredMask} defines the smallest feasible choice of $\tilde\lambda_k$.
    \item \blue{Since $C_1$ depends on the subgaussian norm $K$ of $\X$, the proposed estimator $\SIGMAADAP_n$ is not fully data-adaptive over all possible input distributions. Nevertheless, if one fixes a certain family of modeling distributions with fixed $K$ such as centered Gaussians (as often done in engineering applications), $C_1$ can be approximated using an arbitrary representative of the family and $\SIGMAADAP_n$ becomes fully data-adaptive on this family.}
\end{enumerate}
\end{remark}

Let us now outline the proofs of Theorems~\ref{thm:FrobeniusDitheredMask} and \ref{thm:OperatorDitheredMask} which are detailed in Section \ref{sec:Proofs}. 
Let $\tnorm{\cdot}$ be either the max, Frobenius, or operator norm. We consider the filtration defined by
\begin{equation}
\label{eqn:filtrationDef}
\mathcal{F}_k = \sigma(\X_0,\X_1,\ldots,\X_k,\Tau_1,\ldots,\Tau_k,\bar{\Tau}_1, \ldots,\bar{\Tau}_k), \qquad k=0,1,\ldots,n,
\end{equation}
and let $\E_k(\cdot)=\E(\cdot|\mathcal{F}_k)$ be the associated conditional expectations. The first step is to use the triangle inequality to estimate 
	$$\tnorm{\M \odot \SIGMAADAP_n - \M \odot \SIGMA} \leq \tnorm{\M \odot \SIGMA'_n - \M \odot \SIGMA}$$
	and make the split
	\begin{align}
	\label{eqn:splitExpDithered}
	& \tnorm{\M \odot \SIGMA'_n - \M \odot \SIGMA}\nonumber\\
	& \qquad \le \tnorm{ \M \odot \Big( \SIGMA'_n - \sum_{k=1}^n \E_{k-1}( \tfrac{\tilde\lambda_k^2}{n} \Y_k\bar{\Y}_k^T) \Big) } + \tnorm{ \M \odot \Big( \sum_{k=1}^n \E_{k-1}( \tfrac{\tilde\lambda_k^2}{n} \Y_k\bar{\Y}_k^T) - \SIGMA \Big)},
	\end{align}
where we abbreviate $\Y_k = \sign(\X_k + \Tau_k)$ and $\bar{\Y}_k = \sign(\X_k + \bar{\Tau}_k)$. We can think of the second term on the right-hand side as a `bias term'. The argument to control it is nearly identical for each of the three norms. We will develop this part first in Section~\ref{sec:biasControl}, resulting in Lemma~\ref{lem:biasControlUnified}. The first term on the right-hand side in \eqref{eqn:splitExpDithered} takes the form of (the norm of) a sum of martingale differences. To control this term, we will make use of suitable Burkholder-Rosenthal inequalities. In the Frobenius and max norm cases, it will turn out to be sufficient to use the classical Burkholder-Rosenthal inequalities for real-valued martingale differences. To prove the operator norm bound in Theorem~\ref{thm:OperatorDitheredMask}, however, we will need a new version of the Burkholder-Rosenthal inequalities for matrix martingale differences. We expect this result to be of independent interest.\par 
To formulate this result, recall that a sequence of random matrices $({\XI}_k)_{k=1}^n$ is called a matrix martingale difference sequence with respect to a filtration $(\mathcal{F}_k)_{k=0}^n$ if $\XI_k$ is $\mathcal{F}_k$-measurable, $\E\|\XI_k\| < \infty$, and $\E_{k-1} \XI_k = \0$ for all $1\leq k\leq n$. 
\begin{theorem}
	\label{thm:matrixBRintro} 
	Let $\max\{2,\log n\}\leq q<\infty$ and $\log(p)\geq 2$. There exist $\alpha_{q,p},\beta_{q,p}>0$ depending only on $q$ and $p$ and satisfying $\alpha_{q,p}\lesssim \max\{\sqrt{q},\sqrt{\log(p)}\}$ and $\beta_{q,p}\lesssim q$ if $q\geq \log p$ such that for any matrix martingale difference sequence $({\XI}_k)_{k=1}^n$ in $\R^{p\times p}$ 
	\begin{align}
 \label{eqn:matrixBRintroUpper}
	\Big(\E \Big\|\sum_{k=1}^n {\XI}_k\Big\|^q \Big)^{1/q} & \leq \beta_{q,p}\alpha_{q,p} \max\Big\{\Big(\E\Big\|\Big(\sum_{k=1}^n \mathbb{E}_{k-1}(\XI_k^T\XI_k)\Big)^{1/2}\Big\|^q\Big)^{1/q}, \Big(\E\Big\|\Big(\sum_{k=1}^n \mathbb{E}_{k-1}(\XI_k\XI_k^T)\Big)^{1/2}\Big\|^q\Big)^{1/q},\nonumber\\
	& \qquad \qquad\qquad\qquad\qquad\qquad \qquad\ \ \ \ \ \alpha_{q,p} \max_{1\leq k\leq n} (\E\|{\XI}_k\|^q)^{1/q}\Big\}.
	\end{align}
	Conversely, for any $2\leq q<\infty$, there is a $\gamma_{q,p}>0$ depending only on $q$ and $p$ satisfying $\gamma_{q,p}\lesssim q$ if $q\geq \log(p)$ such that for any matrix martingale difference sequence $({\XI}_k)_{k=1}^n$ in $\R^{p\times p}$ 
		\begin{align}
   \label{eqn:matrixBRintroLower}
	\gamma_{q,p}\Big(\E \Big\|\sum_{k=1}^n {\XI}_k\Big\|^q \Big)^{1/q} & \gtrsim  \max\Big\{\Big(\E\Big\|\Big(\sum_{k=1}^n \mathbb{E}_{k-1}(\XI_k^T\XI_k)\Big)^{1/2}\Big\|^q\Big)^{1/q}, \Big(\E\Big\|\Big(\sum_{k=1}^n \mathbb{E}_{k-1}(\XI_k\XI_k^T)\Big)^{1/2}\Big\|^q\Big)^{1/q},\nonumber\\
	& \qquad \qquad\qquad\qquad\qquad\qquad \qquad\ \ \ \ \ \max_{1\leq k\leq n} (\E\|{\XI}_k\|^q)^{1/q}\Big\}.
	\end{align}
\end{theorem}

The proof of Theorem \ref{thm:matrixBRintro} is presented in Section \ref{sec:Burkholder}.

\subsection{Entry-wise dithering scales}
\label{sec:RankAwareBounds}

\blue{
After the first version of this work appeared, \cite{chen2023parameter} made the interesting observation that improved operator norm estimation error bounds can be obtained by using a different dithering scale for each entry of the samples. Although this was developed for a different quantizer (a two-bit quantizer with triangular dithering), an analogous approach can be pursued in our setting. For simplicity, we consider a setting without an additional mask $\M$.}\par 
\blue{The idea is to consider quantized samples of the form
\begin{align*} 
	\sign(\X_k + \LAMBDA \Tau_k) \quad \text{ and } \quad \sign(\X_k + \LAMBDA \bar{\Tau}_k),
\end{align*}
for $1 \le k \le n$, where $\Tau_k,\bar{\Tau}_k$ are independent and uniformly distributed in $[-1,1]^p$ and $\LAMBDA\in \R^{p\times p}$ is a diagonal matrix. Define the estimator
\begin{align*} 
	(\SIGMA_{\LAMBDA})_n = \frac{1}{2}(\SIGMA_{\LAMBDA}')_n + \frac{1}{2}((\SIGMA_{\LAMBDA}')_n)^T,
\end{align*}
where
\begin{align*} 
	(\SIGMA_{\LAMBDA}')_n = \frac{1}{n} \sum_{k=1}^n  \big( \LAMBDA \sign(\X_k + \LAMBDA \Tau_k) \big) \big( \LAMBDA \sign(\X_k + \LAMBDA \bar{\Tau}_k) \big)^T.
\end{align*}
\sjoerdredred{Before discussing the approach of \cite{chen2023parameter}, let us first observe that if
\begin{equation}
\label{eqn:LambdaOracle}
\LAMBDA_{ii} \simeq_{K} \sqrt{\log(n) \SIGMA_{ii}},
\end{equation}
then the proof of Theorem~\ref{thm:OperatorDitheredMask_Old} can be readily generalized to show that, with high probability, 
\begin{equation}
\label{eqn:sigmaLambdaEst}
\|(\SIGMA_{\LAMBDA})_n - \SIGMA\|\lesssim_K \text{polylog}(p)\| \SIGMA \| \left( \sqrt{\frac{\operatorname{r}(\SIGMA)}{n}} +\frac{\operatorname{r}(\SIGMA)}{n} \right),
\end{equation}
where $\operatorname{r}(\SIGMA)=\tr(\SIGMA)/\|\SIGMA\|$ is the \emph{effective rank} of $\SIGMA$. This bound is qualitatively better than the ones in Theorem~\ref{thm:OperatorDitheredMask_Old}, as clearly $\tr(\SIGMA)\leq p\|\SIGMA\|_{\infty}$. In fact, this error estimate is the same as the (sharp) error bound for the sample covariance matrix of the samples \emph{before quantization} \cite{koltchinskii2017}. On the other hand, compared to Theorem~\ref{thm:OperatorDitheredMask_Old} one needs prior knowledge of all diagonal entries of $\SIGMA$, rather than just (an upper bound on) the largest variance.}\par
\sjoerdredred{Rather than the oracle-type variant \eqref{eqn:LambdaOracle}, \cite{chen2023parameter} directly proposed to use the data-driven choice $\LAMBDA=\LAMBDA_{\max}$, where
$$(\LAMBDA_{\max})_{ii} = \max_{1\leq k\leq n}|(\X_k)_i|.$$
}}\sjoerdredred{As is already indicated in \cite{chen2023parameter}, similarly to their analysis (developed for a different two-bit quantizer with triangular dither), one can show that with this choice \eqref{eqn:sigmaLambdaEst} again holds with high probability. Although this approach leads to a superior operator norm estimate compared to Theorem~\ref{thm:OperatorDitheredMask}, note that it relies on much more side information. Indeed, the dithering levels in \eqref{eq:TwoBitSamplesAdaptive} are \emph{predictable quantities}, meaning that the dithers for $\X_k$ are based on quantities depending only on $\X_1,\ldots,\X_{k-1}$. This enables quantization in the typical signal processing scenario where samples are collected sequentially and need to be quantized as they arrive at a sensor. In contrast, $\Lambda_{\max}$ depends on all collected samples, hence \emph{all samples $\X_1,\ldots,\X_n$ must be observed and stored before they can be quantized}. This prohibits the aforementioned sequential sensing scenario. Moreover, even if one is able to obtain all samples before quantization, then it seems that one may as well construct the sample covariance matrix of these unquantized samples as an estimator (which satisfies the same estimation error bound).\par 
In the remainder of this section, we will show that the two approaches can be reconciled: we will show that improved error estimates can still be obtained by using \emph{predictable} entry-wise dithering levels.}
\blue{Specifically, consider now quantized samples of the form
\begin{align} \label{eq:TwoBitSamplesAdaptiveEntrywise}
	\sign(\X_k + \LAMBDA_k \Tau_k) \quad \text{ and } \quad \sign(\X_k + \LAMBDA_k \bar{\Tau}_k),
\end{align}
for $1 \le k \le n$, where $\Tau_k,\bar{\Tau}_k$ are \newS{again} independent and uniformly distributed in $[-1,1]^p$ and $\LAMBDA_k$ is defined by its diagonal $\diag(\LAMBDA_k) \in \R_{>0}^p$. We set $\diag(\LAMBDA_k) = c \log(k) \cdot \LAmbda_k$, for $c > 0$ an absolute constant (to be specified later) and $\LAmbda_k$ defined by 
\begin{align} \label{eq:lambdaAdaptiveEntrywise}
	\LAmbda_{k}
	= \left( \frac{1}{k} \sum_{j=0}^{k-1} \X_j^{\odot 2} \right)^{\odot 1/2}
	= \left( \frac{1}{k} \left( (k-1) \LAmbda_{k-1}^{\odot 2} + \X_{k-1}^{\odot 2} \right) \right)^{\odot 1/2}, \qquad 1\leq k\leq n,
\end{align}
where $\X_j^{\odot \ell} = \X_j \odot \cdots \odot \X_j$ denotes the entry-wise power of a vector.
Now consider the estimator
\begin{align} \label{eq:TwoBitEstimatorAdaptiveEntrywise}
	\newS{\hat{\SIGMA}^{\operatorname{adap}}_n} = \frac{1}{2}\hat{\SIGMA}'_n + \frac{1}{2}(\hat{\SIGMA}'_n)^T,
\end{align}
where
\begin{align} \label{eq:AsymmetricEstimatorAdaptiveEntrywise}
	\newS{\hat{\SIGMA}'_n} = \frac{1}{n} \sum_{k=1}^n  \big( \LAMBDA_k \sign(\X_k + \LAMBDA_k\Tau_k) \big) \big( \LAMBDA_k \sign(\X_k + \LAMBDA_k \bar{\Tau}_k) \big)^T.
\end{align}
By slightly adapting the proof of Theorem \ref{thm:OperatorDitheredMask}, we can derive the following effective rank-aware bounds for our data-adaptive estimator. \sjoerdredred{We provide the full proof in Appendix~\ref{sec:ProofAppendix}.}
%
\begin{theorem} \label{thm:OperatorDitheredMaskDecoupled}
	For any $K > 0$, there exist constants $C_1,C_2>0$ depending only on $K$ such that the following holds. Let $\X \in \R^p$ be a mean-zero, $K$-subgaussian vector with covariance matrix \sjoerd{$\SIGMA \in \R^{p\times p}$}.
	Let $\X_0,...,\X_n \overset{\mathrm{i.i.d.}}{\sim} \X$ and let $\M \in [0,1]^{p\times p}$ be a fixed symmetric mask. If $\LAMBDA_k \in \R^{p\times p}$ are chosen to be diagonal with $\diag(\LAMBDA_k) = C_1 \log(k)^{1/2} \LAmbda_k$, where $\LAmbda_k$ is defined in \eqref{eq:lambdaAdaptiveEntrywise}, then for any $\theta \ge C_2\max\{\log(p),\log(n)\}$ we have with probability at least $1 - 4e^{-\theta}$ $$\pnorm{\newS{\hat{\SIGMA}^{\operatorname{adap}}_n} -  \SIGMA}{}\lesssim_{K} 
	\log(p)\log(n) \| \SIGMA \| \left(\theta^{5/2} \sqrt{\frac{\newS{\operatorname{r}(\SIGMA)}}{n}} +\theta^{7/2} \frac{\newS{\operatorname{r}(\SIGMA)}}{n} \right).
	$$  
\end{theorem}
\begin{remark}
	\label{rem:Decoupled}
Using a different dithering level for each entry seems to be more of theoretical than of practical interest since it massively increases acquisition and storage complexity of the samples. In contrast to using a global dithering scale, one has to acquire a full-precision number per single entry and sample vector, and to keep track of the diagonal matrix $\LAMBDA_k$ instead of the single parameter $\lambda_k$. In standard $32$-bit representation this means that one has total storage costs of $32p + 2pn$ bits. On the one hand, this still improves over the total storage costs of $32pn$ bits required in the case of high-precision samples. On the other hand, to update $\LAMBDA_k$ the decoupled dithering procedure acquires one full precision number per entry of each sample. In physical systems, in which sample entries correspond to single measurement devices with distinct analog-to-digital converters, the decoupled dithering procedure has thus the same acquisition complexity as \newS{directly} using full precision samples. 
\end{remark}
}

\section{Numerical simulation}
\label{sec:Numerics}

\blue{
\johannesnew{To illustrate Theorem~\ref{thm:OperatorDitheredMask} numerically\footnote{The code is provided at \url{https://johannes-maly.github.io}.}, we revisit a simulation in \cite[Figure 1(a)]{dirksen2021covariance}. We consider a Gaussian distribution and covariance matrix $\SIGMA$ having ones on its diagonal and all remaining entries equal to $0.2$. We plot the average estimation errors of the two estimators in operator norm, for $n = 200$ and $p$ varying from $5$ to $30$ averaged over $100$ draws of $\X_0,\dots,\X_n$. The dithered estimator $\SIGMADITH_n$ uses $\lambda \in (0,4\| \SIGMA \|_\infty)$ that is optimized in each run via grid-search. The constant $C_1$ from Theorem \ref{thm:OperatorDitheredMask}, which is required for $\SIGMAADAP_n$, is optimized only in the first run for $p=5$ via grid-search (yielding $C_1 = 0.20$) and fixed for the rest of the experiment. As expected, both estimators perform similarly, with $\SIGMADITH_n$ showing slightly better performance. Importantly, however, this slightly better performance only holds if $\lambda$ is (nearly) perfectly tuned. \sjoerdredred{In practice, one typically does not have prior knowledge of $\|\SIGMA\|_{\infty}$ and hence one cannot perfectly tune $\lambda$.} As can be seen in Figure~\ref{fig:VaryingLambda1}, $\SIGMAADAP_n$ significantly outperforms $\SIGMADITH_n$ if $\lambda$ is imperfectly tuned. Finally, let us observe that the constant $C_1$ does not need to be carefully tuned - in Theorem~\ref{thm:OperatorDitheredMask} it depends only on the subgaussian constant $K$ and, in particular, is a universal constant if we consider Gaussian samples. To further verify this numerically, we test the choice of $C_1$ from the first experiment for two different, randomly drawn covariance matrices. As seen in Figures \ref{fig:VaryingLambda2}-\ref{fig:VaryingLambda3}, the performance of $\SIGMAADAP_n$ is steady with this choice.} 
%
\begin{figure}
    \centering
    \includegraphics[width=0.5\textwidth]{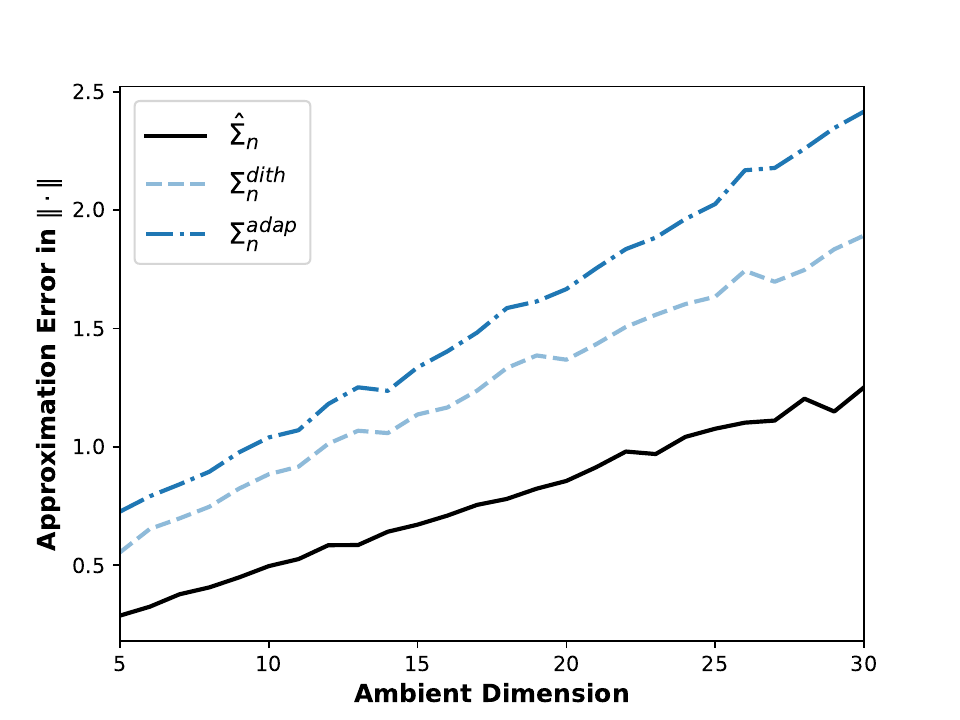}
    \caption{\johannesnew{Comparison of $\SIGMAADAP_n$ and $\SIGMADITH_n$ (with optimized $\lambda$). The performance of the sample covariance matrix $\hat\SIGMA_n = \frac{1}{n} \sum_{k=1}^n \X_k\X_k^T$, which uses unquantized samples, is plotted for reference.} \sjoerdredred{It is remarkable that $\SIGMAADAP_n$ performs almost as well as the oracle-based estimator $\SIGMADITH_n$ with an optimized $\lambda$.}}
    \label{fig:Comparison}
\end{figure}
\begin{figure}
    \centering
     \begin{subfigure}[b]{0.45\textwidth}
         \centering
         \includegraphics[width=\textwidth]{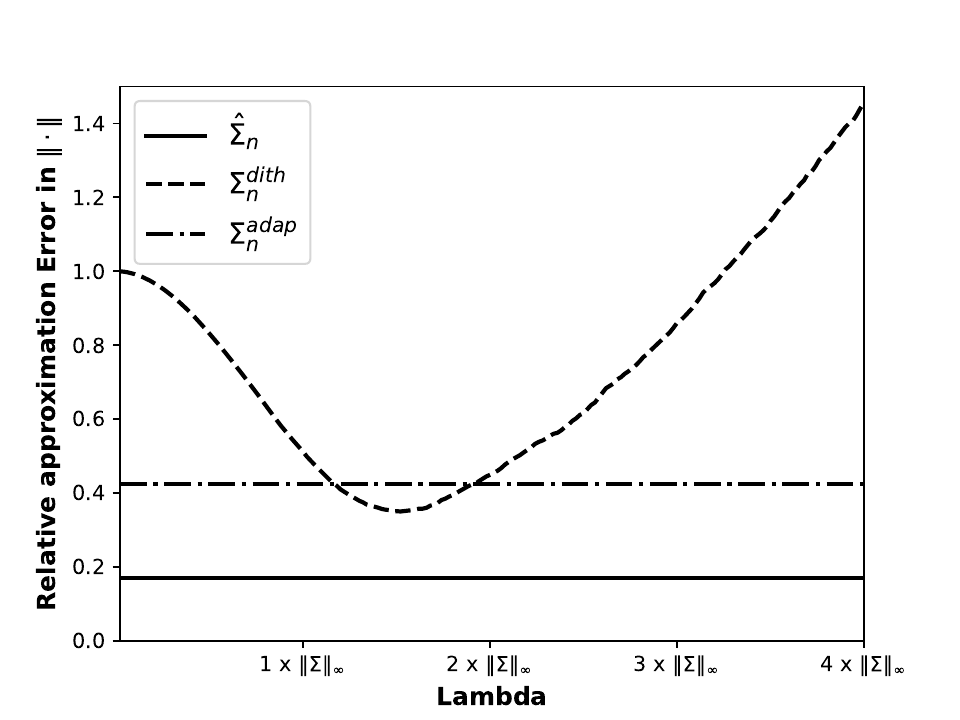}
         \caption{$\SIGMA$ as in Figure \ref{fig:Comparison}.}
         \label{fig:VaryingLambda1}
     \end{subfigure}
     \\ \vspace{0.5cm}
     \begin{subfigure}[b]{0.45\textwidth}
         \centering
         \includegraphics[width=\textwidth]{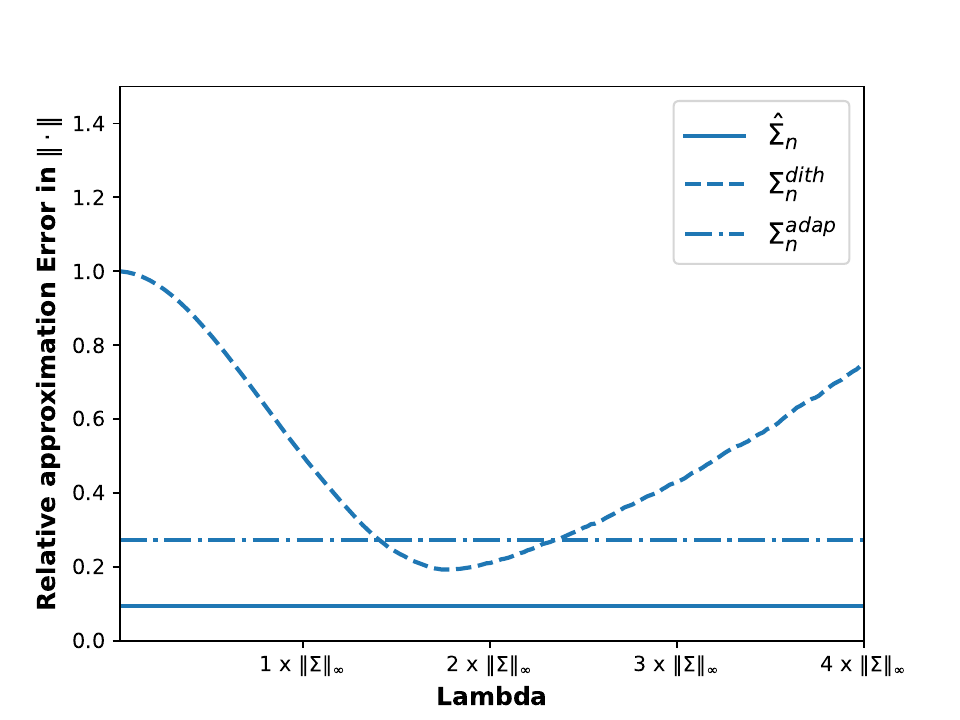}
         \caption{Random $\SIGMA$.}
         \label{fig:VaryingLambda2}
     \end{subfigure}
     \hfill
     \begin{subfigure}[b]{0.45\textwidth}
         \centering
         \includegraphics[width=\textwidth]{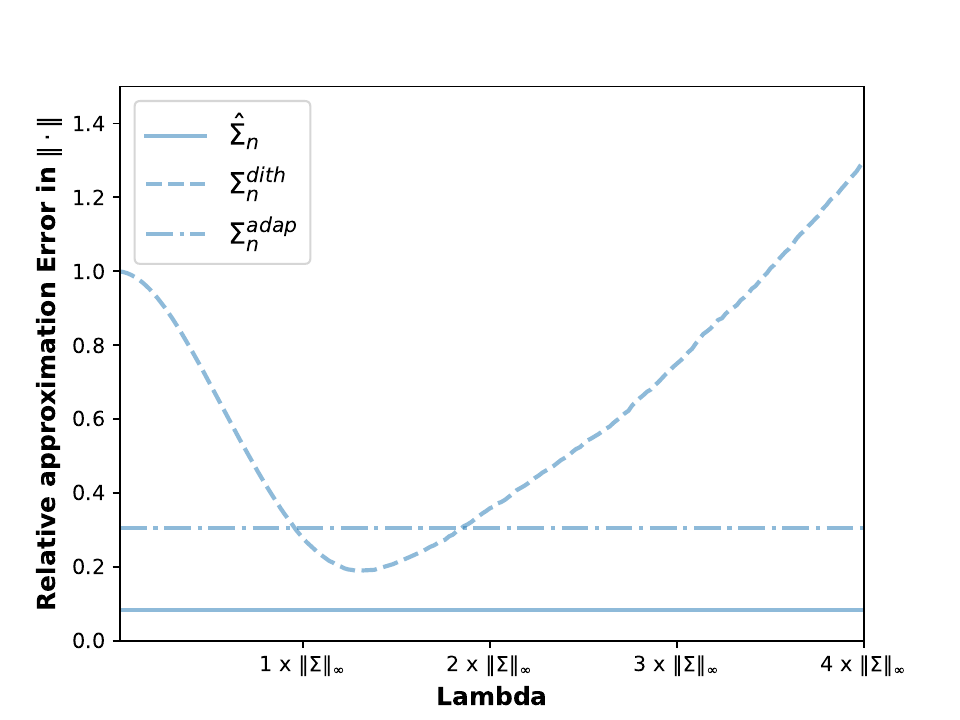}
         \caption{Random $\SIGMA$.}
         \label{fig:VaryingLambda3}
     \end{subfigure}
     \caption{\johannesnew{Performance of the two estimators 
     while varying $\lambda$ in $\SIGMADITH_n$. The constant $C_1$ is chosen as in Figure \ref{fig:Comparison}. Figure \ref{fig:VaryingLambda1} considers the same $\SIGMA$ as in Figure \ref{fig:Comparison} for $p=5$, whereas Figures \ref{fig:VaryingLambda2}-\ref{fig:VaryingLambda3} consider two randomly drawn covariance matrices. In all cases, the performance of $\SIGMAADAP_n$ is close to the optimal performance of $\SIGMADITH_n$, even though $C_1$ is fixed to the same value. Note that in contrast to Figure \ref{fig:Comparison}, the error is measured relatively to $\| \SIGMA \|$ to allow easier comparison. The performance of the sample covariance matrix $\hat\SIGMA_n$ is plotted for reference. }} 
    \label{fig:VaryingLambda}
\end{figure}
}

\blue{
To illustrate Theorem \ref{thm:OperatorDitheredMaskDecoupled} numerically as well, we compare the performance of $\SIGMAADAP_n$ and $\hat{\SIGMA}^{\operatorname{adap}}_n$ in the setting of Figure \ref{fig:Comparison} for two different covariance matrices. The first covariance matrix $\SIGMA_1$ is the one from Figure \ref{fig:Comparison}, i.e., it has ones on the diagonal and all remaining entries equal to $0.2$; clearly, $\tr(\SIGMA_1) = p \| \SIGMA_1 \|_\infty$ so that Theorems \ref{thm:OperatorDitheredMask} and \ref{thm:OperatorDitheredMaskDecoupled} yield similar bounds. This is \sjoerdredred{reflected in} the simulation. The second covariance matrix $\SIGMA_2 = \mathbf D \SIGMA_1 \mathbf D$, where $\mathbf D \in \R^{p\times p}$ is a diagonal matrix with $D_{11} = 1$ and $D_{ii} = 1/10$, for $i > 1$; in this case, all variances but the first are scaled by $\tfrac{1}{100}$ so that $\tr(\SIGMA_2) \ll p \| \SIGMA_2 \|_\infty$ and Theorem \ref{thm:OperatorDitheredMaskDecoupled} yields \sjoerdredred{ a much better error estimate} for $\hat{\SIGMA}^{\operatorname{adap}}_n$ than Theorem \ref{thm:OperatorDitheredMask} for $\SIGMAADAP_n$. Figure \ref{fig:ComparisonDecoupled} illustrates that in the second case the error of $\SIGMAADAP_n$ grows in $p$ whereas the error of $\hat{\SIGMA}^{\operatorname{adap}}_n$ stays constant and aligns with the behavior of the sample covariance matrix $\hat\SIGMA_n$ \sjoerdredred{of the} unquantized samples.
\begin{figure}
    \centering
    \includegraphics[width=0.6\textwidth]{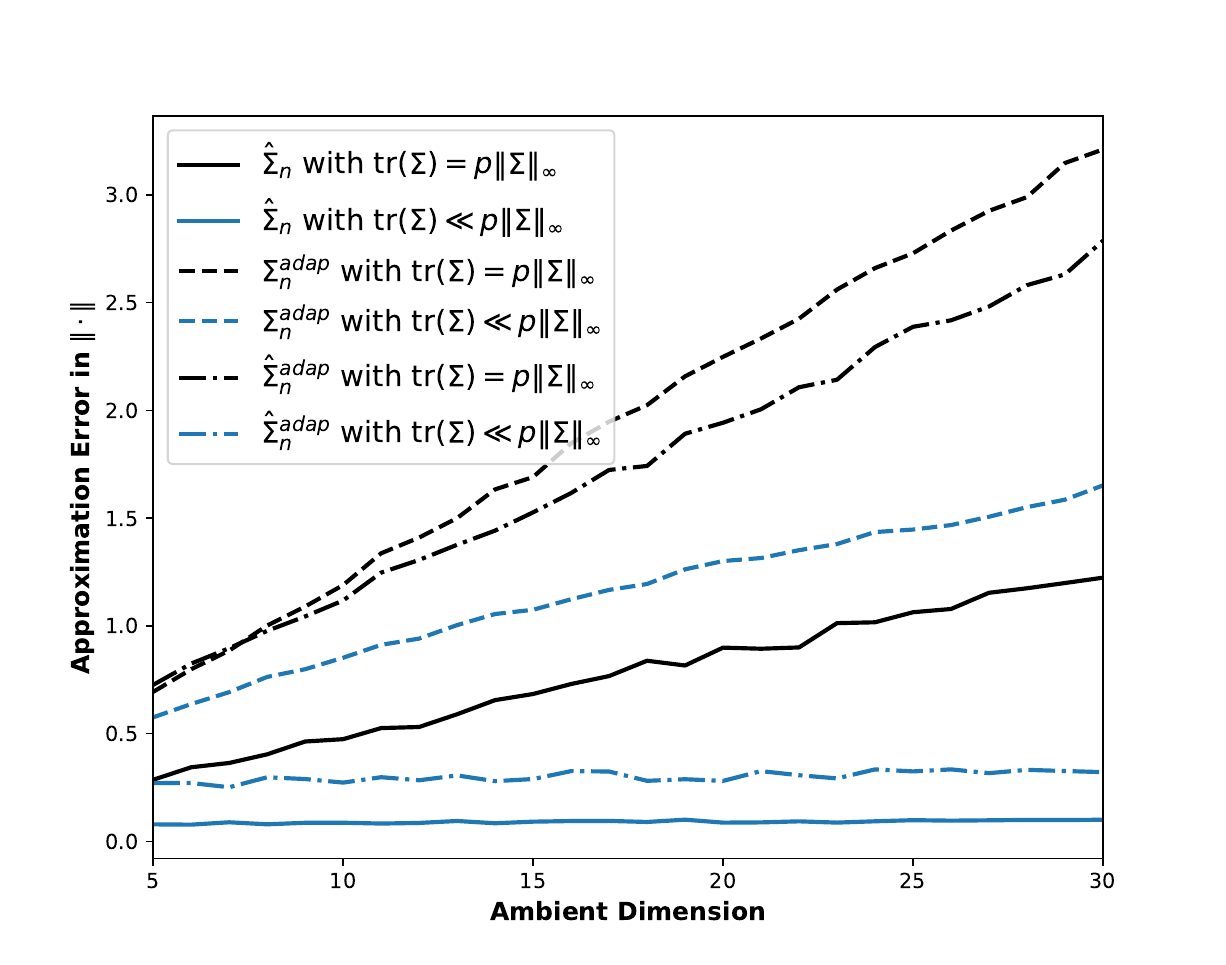}
    \caption{\newJ{Comparison of $\SIGMAADAP_n$ and $\hat{\SIGMA}^{\operatorname{adap}}_n$. The performance of the sample covariance matrix $\hat\SIGMA_n = \frac{1}{n} \sum_{k=1}^n \X_k\X_k^T$, which uses unquantized samples, is plotted for reference.}}
    \label{fig:ComparisonDecoupled}
\end{figure}
}

\section{Related work}
\label{sec:RelatedWork}
\blue{

The observation that dithering, the operation of intentionally adding well-designed `noise' before quantization, is advantageous for signal reconstruction has a long history going back at least to \cite{Rob62} (see also the survey \cite{gray1998quantization}). In the last decade, there have been various works that rigorously quantify this advantage for scenarios involving the memoryless one-bit scalar quantizer $Q(\X) = \sign(\X)$ featured in this work. For instance, there is now a substantial literature on one-bit compressed sensing with dithering. In the context of one-bit compressed sensing with zero thresholds, it is well-known that it is only possible to reconstruct the direction of a signal and, moreover, positive reconstruction results are restricted to Gaussian and Gaussian circulant measurements and reconstruction can easily fail for other measurement schemes (see, e.g., (the references in) the works \cite{ALP14,BFN17,BoB08,DJR20,Fou17,JLB13,PlV13,PlV13lin} and the surveys \cite{BJK15,Dir19}). It was first observed that Gaussian dithering can be leveraged to reconstruct the full signal (see \cite{BFN17,knudson2016one} for Gaussian measurements and \cite{DJR20} for Gaussian circulant measurements). Later, it was shown that uniform dithering enables reconstruction in sensing scenarios beyond Gaussian(-like) matrices, including subgaussian and heavy-tailed measurements (e.g., \cite{DiM18a,JMP19}), subgaussian circulant measurements \cite{DiM23}, and DCT matrices \cite{EYN23b}, using efficient reconstruction methods. Moreover, reconstruction can be made optimally robust against measurement noise prior to quantization as well as adversarial bit corruptions occurring during quantization \cite{DiM18a,DiM23}. The advantage of using dithering in the memoryless one-bit quantizer has been explored for a range of other problems in signal processing, including one-bit matrix completion (e.g., \cite{CaZ13,chen2022high,DPB14,EYS23}), reconstructing a signal in an unlimited sampling framework with one-bit quantization (e.g., \cite{EMY22,EMY23}), and one-bit quadratic sensing problems such as phase retrieval (e.g., \cite{EYN23,EYS22}). Below we will only give a detailed account of the developments in our problem setting, covariance estimation.}    

\subsection{Quantized covariance estimation}

Engineers have been examining the influence of coarse quantization on correlation and covariance estimation for decades, e.g., in the context of signal processing \cite{jacovitti1994estimation,roth2015covariance}, direction of arrival (DOA) estimation \cite{bar2002doa}, and massive multiple-input-multiple-output (MIMO) \cite{choi2016near,li2017channel}. While many related works restrict themselves to the undithered one-bit quantizer $Q(\X) = \sign(\X)$ and use the so-called arcsine law \cite{van1966spectrum} (also known as Grothendieck's identity in the mathematical literature \cite{vershynin2018high}) for estimating the correlation matrix, some recent articles examine the use of quantization thresholds in order to estimate the full covariance matrix.

One line of work focuses on the covariance estimation of stationary \cite{eamaz2021modified,eamaz2023covariance} and non-stationary \cite{eamaz2022covariance} Gaussian processes. In order to estimate the full covariance matrix, the authors propose to add biased Gaussian dithering vectors before quantizing. Due to the Gaussian nature of the dithering thresholds, the authors can derive a modified arcsine law which relates the covariance matrix of the quantized samples $\E(\sign(\X+\Tau)\sign(\X+\Tau)^T)$ to an involved integral expression depending on the entries of $\E(\X\X^T)$. To reconstruct $\E(\X\X^T)$ from $\E(\sign(\X+\Tau)\sign(\X+\Tau)^T)$, different numerical methods are proposed. Note, however, that the modified arcsine law is tailored to Gaussian samples and that none of the three works analyzes the estimation error in terms of the number of samples. In fact, the modified arcsine law only applies to the sample covariance matrix in the infinite sample limit ($n \to \infty$).

Following the observation that variances can be estimated from one-bit samples with fixed, deterministic quantization thresholds \cite{chapeau2008fisher,fang2010adaptive}, another line of work investigates the use of such fixed thresholds \cite{liu2021one} to estimate covariance matrices from one-bit samples. Assuming an underlying stationary Gaussian process, the authors explicitly calculate the mean and autocorrelation of one-bit samples with fixed dither. Based on these representations they design an algorithm for estimating the covariance matrix entries of the underlying process and empirically evaluate its performance. Again, the theoretical results are tailored to Gaussian samples and do not analyze the approximation error in terms of the number of samples. Moreover, the provided theoretical results are restricted to single entries of the covariance matrix. Since this does not take into account the matrix structure of $\E(\X\X^T)$, the developed tools will not lend themselves to a tight analysis of the approximation error in matrix Schatten norms like the operator norm.

The authors of \cite{xiao2023one} follow up on \cite{liu2021one} and use low-order Taylor expansions to analyze the variance of maximum-likelihood estimators of the quantities computed in \cite{liu2021one}. They conclude that a fixed dither yields suboptimal results if the variances of the process are strongly fluctuating. They then propose to use a dither that is monotonically increasing in time and evaluate its performance in numerical simulations. As with the previous works, the applied techniques are tailored to Gaussian sample distributions and the analysis is restricted to the infinite sample limit ($n \to \infty$).

In contrast, the results in \cite{dirksen2021covariance}, which we extend in this manuscript, provided the first non-asymptotic (and near-minimax optimal) guarantees in the literature. The use of uniform random dithering vectors (instead of fixed, deterministically varying, or random Gaussian ones) considerably simplifies the analysis and allows to generalize the results to non-Gaussian samples. In \cite{yang2023plug}, the non-asymptotic results of \cite{dirksen2021covariance} have been generalized to the complex domain and applied to massive MIMO. A recent work \cite{chen2022high} modified the strategy of \cite{dirksen2021covariance} to cover heavy-tailed distributions (by using truncation before quantizing). \blue{The work  \cite{chen2023parameter}, which appeared after the first version of our work, introduced a new covariance estimator based on samples obtained using a two-bit quantizer with triangular (rather than uniform) dithering. This paper derived operator norm error estimates for this estimator, both for a single dithering level proportional to the maximal variance (similar to Theorem~\ref{thm:OperatorDitheredMask_Old}) and data dependent dithering levels (for the latter, see the detailed discussion in section~\ref{sec:RankAwareBounds}).} 

\blue{
\subsection{Matrix martingale inequalities}

The classical Burkholder-Rosenthal inequalities for real-valued martingales (see Theorem~\ref{thm:BRreal}) are due to Burkholder \cite{Bur73}, following the work of Rosenthal~\cite{Ros70}. These inequalities were later extended to noncommutative martingales by Junge and Xu \cite{JuX03}. For matrix martingales, their result implies that for $2\leq q<\infty$
 \begin{align}
 \label{eqn:JX}
\Big(\E \Big\|\sum_{k=1}^n {\XI}_k\Big\|_{S^q}^q \Big)^{1/q} & \leq C_q \max\Big\{\Big(\E\Big\|\Big(\sum_{k=1}^n \mathbb{E}_{k-1}(\XI_k^T\XI_k)\Big)^{1/2}\Big\|_{S^q}^q\Big)^{1/q}, \Big(\E\Big\|\Big(\sum_{k=1}^n \mathbb{E}_{k-1}(\XI_k\XI_k^T)\Big)^{1/2}\Big\|_{S^q}^q\Big)^{1/q},\nonumber\\
	& \qquad \qquad\qquad\qquad\qquad\qquad \qquad \Big(\sum_{k=1}^n \E\|{\XI}_k\|_{S^q}^q\Big)^{1/q}\Big\},
\end{align}
where $\|\cdot\|_{S^q}$ denotes the Schatten $q$-norm and $C_q$ is a constant depending only on $q$. Note that in \eqref{eqn:JX}, the $L^q$-norm over the probability space matches the parameter of the Schatten norm. In contrast, Theorem~\ref{thm:matrixBRintro} states a different generalization of the Burkholder-Rosenthal inequalities for $(\E\|\sum_{k=1}^n {\XI}_k\|_{S^{\infty}}^q)^{1/q}$. Our proof can be readily adapted to yield an estimate for any `mixed norm' $(\E\|\sum_{k=1}^n {\XI}_k\|_{S^{r}}^q)^{1/q}$, with $r>1$, of the matrix martingale.\par
Let us note that two other well-known existing inequalities for matrix martingales, the matrix versions of the Azuma-Hoeffding and Burkholder-Gundy inequalities, are insufficient for our purposes. First, it is important that Theorem~\ref{thm:matrixBRintro} allows to estimate \emph{unbounded} matrix martingales, whereas the matrix version of the Azuma-Hoeffding inequality \cite{tropp2012user} requires uniformly bounded martingale differences. Second, in our proof it is essential that the right hand side of \eqref{eqn:matrixBRintroUpper} features $L^q$-norms of \emph{predictable}, i.e., $\mathcal{F}_{k-1}$-measurable quantities. In contrast, the well-known noncommutative version of the Burkholder-Gundy inequalities \cite{PiX97} implies estimates for matrix martingales with a non-predictable right hand side:
\begin{align}
 \label{eqn:PX}
\Big(\E \Big\|\sum_{k=1}^n {\XI}_k\Big\|_{S^q}^q \Big)^{1/q} & \leq C_q \max\Big\{\Big(\E\Big\|\Big(\sum_{k=1}^n \XI_k^T\XI_k\Big)^{1/2}\Big\|_{S^q}^q\Big)^{1/q}, \Big(\E\Big\|\Big(\sum_{k=1}^n \XI_k\XI_k^T\Big)^{1/2}\Big\|_{S^q}^q\Big)^{1/q},
\end{align}
where $C_q$ is a constant depending only on $q$.}

\section{Proofs}
\label{sec:Proofs}

As outlined above, we first analyze the bias term of \eqref{eqn:splitExpDithered} in Section \ref{sec:biasControl}. This will result in Lemma \ref{lem:biasControlUnified} below. Sections \ref{sec:ProofOfFrobeniusDitheredMask} and \ref{sec:ProofOfOperatorDitheredMask} then present the proofs of Theorems \ref{thm:FrobeniusDitheredMask} and \ref{thm:OperatorDitheredMask}. Finally, Section \ref{sec:Burkholder} provides the proof of Theorem \ref{thm:matrixBRintro}. 

\subsection{Control of the bias}
\label{sec:biasControl}

Let us begin with some basic observations. First note that since $\lambda_k$ in \eqref{eq:lambdaAdaptive} --- and with it $\tilde\lambda_k = c \log(k) \lambda_k$ --- is independent of $\X_k$, we can apply the following result from \cite{dirksen2021covariance} conditionally to estimate the bias term in \eqref{eqn:splitExpDithered}.

\begin{lemma}[{\cite[Lemma 17]{dirksen2021covariance}}]
	\label{lem:linftyBiasEst}
	For any $K > 0$, there exist constants $c_1,c_2>0$ depending only on $K$ such that the following holds. Let $\X \in \R^p$ be a mean-zero, $K$-subgaussian vector with covariance matrix $\johannes{\E\round{ \X\X^T }}{} = \SIGMA \in \R^{p\times p}$. Let \sjoerd{$\lambda>0$ and let} $\Y=\sign(\X + \Tau)$ and $\bar{\Y}=\sign(\X + \bar{\Tau})$, where $\Tau,\bar{\Tau}$ are independent and uniformly distributed in $[-\lambda,\lambda]^p$ \sjoerd{and independent of $\X$}. Then,
	$$\pnorm{\johannes{\E\round{\lambda^2\Y\bar{\Y}^T}}{} - \SIGMA}{\infty}\sjoerd{\leq c_1} (\lambda^2+\|\SIGMA\|_{\infty})e^{-\sjoerd{c_2}\lambda^2/\|\SIGMA\|_{\infty}}$$
	and 
	$$\pnorm{\johannes{\E\round{\lambda^2\Y\Y^T}}{} - (\SIGMA-\diag(\SIGMA)+\lambda^2 \id)}{\infty}\sjoerd{\leq c_1} (\lambda^2+\|\SIGMA\|_{\infty})e^{-\sjoerd{c_2}\lambda^2/\|\SIGMA\|_{\infty}}.$$
\end{lemma}

To make use of the bias estimate in Lemma \ref{lem:linftyBiasEst}, however, we have to control the adaptive dithering ranges~$\tilde\lambda_k$. This is achieved in the following lemma.

\begin{lemma}
\label{lem:linfConcentration}
    There exist absolute constants $c,C>0$ and, for any given $K > 0$, an $L\gtrsim K^{-3}$ such that the following holds. Let $\X \in \R^p$ be a $K$-subgaussian vector with $\E \X = \0$ and $\E(\X\X^\top) = \SIGMA \in \R^{p\times p}$. Let $\X_0,\dots,\X_n \overset{i.i.d.}{\sim} \X$ and let $\lambda_k$ be defined as in \eqref{eq:lambdaAdaptive}. Then, for $t > 0$,
    \begin{align}
    \label{eqn:lambdakSubgaussian}
        \P{\lambda_{k} > (t+1) CK \sqrt{\| \SIGMA \|_\infty \log(2p)}} \le 2 e^{-ckt^2}
    \end{align}
    and 
    \begin{align}
    \label{eqn:lambdakHPLowerBd}          \P{\lambda_{k} < (1-t) L \sqrt{\| \SIGMA \|_\infty} } \le 2 e^{ - \frac{ckL^2t^2}{K^2\log(p) } }.
    \end{align}
\end{lemma}

In order to prove Lemma \ref{lem:linfConcentration}, we control the $\ell_\infty$-norm of a subgaussian vector as follows.

\begin{lemma}
\label{lem:linfExpectation}
There exists an absolute constant $C>0$ such that the following holds. Let $\X \in \R^p$ be a $K$-subgaussian vector with $\E \X = \0$ and $\E(\X\X^\top) = \SIGMA \in \R^{p\times p}$. Then $\| \X \|_\infty$ is $(CK \sqrt{ \| \SIGMA \|_\infty \log(p)})$-subgaussian and there exists an $L \gtrsim K^{-3}$ such that
    \begin{align*}
      L \sqrt{ \| \SIGMA \|_\infty} \le  \E \| \X \|_\infty \lesssim K \sqrt{\| \SIGMA \|_\infty \log(2p)}.
    \end{align*}
\end{lemma}

\begin{proof}
    If $\X$ is $K$-subgaussian, all entries $X_i$ of $\X$ satisfy by definition 
    \begin{align}
    \label{eq:psi2-normComponent}
        \| X_i \|_{\psi_2} = \| \langle \X, \e_i \rangle \|_{\psi_2} \le K ( \E \langle \X,\e_i \rangle^2 )^\frac{1}{2} \le K \sqrt{\| \SIGMA \|_\infty}.
    \end{align}
    Thus we get for any $t > 0$ that
    \begin{align*}
        \P{\| \X \|_\infty - C'K\sqrt{\| \SIGMA \|_\infty \log(p)} > t}
        &= \P{ \exists i \in [p] \colon |X_i| > t + C' K\sqrt{\| \SIGMA \|_\infty \log(p)}}
        \\
        &\le \sum_{i=1}^p \P{|X_i| > t + C'K\sqrt{\| \SIGMA \|_\infty \log(p)}} \\
        &\le 2p e^{-\frac{c(t + C' K\sqrt{\| \SIGMA \|_\infty \log(p)})^2}{K^2 \| \SIGMA \|_\infty}}
        \le 2e^{-\frac{ct^2}{K^2 \| \SIGMA \|_\infty}} e^{-\frac{2 t KcC'\sqrt{\| \SIGMA \|_\infty \log(p)}}{K^2 \| \SIGMA \|_\infty}}
        \\
        &\le 2e^{-c\frac{t^2}{K^2 \| \SIGMA \|_\infty}} \\
        &\le 2e^{-c\frac{t^2}{K^2 \| \SIGMA \|_\infty \log(p)}},
    \end{align*}
   provided that $C'$ is a large enough absolute constant. Together with 
   $$\| \X \|_\infty - K\sqrt{\| \SIGMA \|_\infty \log(p)} \ge - K\sqrt{\| \SIGMA \|_\infty \log(p)}$$ 
   this yields 
    \begin{align*}
        \| (\| \X \|_\infty - K\sqrt{\| \SIGMA \|_\infty \log(p)}) \|_{\psi_2} \lesssim K\sqrt{\| \SIGMA \|_\infty \log(p)},
    \end{align*}
    see e.g.\ \cite[Section 2.5]{vershynin2018high}, and thus the first claim via
    \begin{align*}
        \| \| \X \|_\infty \|_{\psi_2} &\le \| (\| \X \|_\infty - K\sqrt{\| \SIGMA \|_\infty \log(p)}) \|_{\psi_2} + \| K\sqrt{\| \SIGMA \|_\infty \log(p)} \|_{\psi_2} \\
        &\lesssim K\sqrt{\| \SIGMA \|_\infty \log(p)}.
    \end{align*}
    
    Since $\E\X = \0$ by assumption, \eqref{eq:psi2-normComponent} further implies $\E \exp(\theta X_i) \le \exp(CK^2 \| \SIGMA \|_\infty \theta^2)$, for some absolute constant $C > 0$ and any $\theta \in \R$, see e.g.\ \cite[Section 2.5]{vershynin2018high}, so that \cite[Proposition 7.29]{foucart2013compressed} yields the upper bound on $\E \|\X\|_\infty$ via
    \begin{align*}
        \E \| \X \|_\infty 
        = \E \left( \max_{i \in [p]} |X_i| \right)
        \lesssim K \sqrt{\| \SIGMA \|_\infty \log(2p)}.
    \end{align*}
    To see the lower bound, let $i_\star \in [p]$ denote an index such that $X_{i_\star}$ has maximal variance, i.e., $\Sigma_{i_\star, i_\star} = \| \SIGMA \|_\infty$. Then, 
    $$\|X_{i_\star}\|_{L^2} = \|\langle \X, \e_{i_\star} \rangle\|_{L^2} =  \sqrt{\| \SIGMA \|_\infty}.$$
    On the other hand, by H\"older's inequality and \eqref{eq:psi2-normComponent},
    $$\|X_{i_\star}\|_{L^2}\leq \|X_{i_\star}\|_{L^1}^{1/4}\|X_{i_\star}\|_{L^3}^{3/4}\lesssim \|X_{i_\star}\|_{L^1}^{1/4} (K \sqrt{\| \SIGMA \|_\infty})^{3/4}$$
    so that
    $$\|X_{i_\star}\|_{L^1} \gtrsim K^{-3}\sqrt{\| \SIGMA \|_\infty}$$
    and hence
    $$\E \| \X \|_\infty \geq \|X_{i_\star}\|_{L^1} \gtrsim K^{-3}\sqrt{\| \SIGMA \|_\infty}.$$
    
%
\end{proof}

\begin{proof}[Proof of Lemma \ref{lem:linfConcentration}]
    Recall that 
    \begin{align*}
        \lambda_{k}
        = \frac{1}{k} \sum_{j=0}^{k-1} \| \X_j \|_\infty.
    \end{align*}
    Since $\| \X \|_\infty$ is $(CK \sqrt{ \| \SIGMA \|_\infty \log(p)})$-subgaussian by Lemma \ref{lem:linfExpectation}, Hoeffding's inequality \cite[Theorem 2.6.2]{vershynin2018high} yields
    \begin{align*}
        \P{|\lambda_{k} - \E \| \X \|_\infty | > u}
        = \P{ \Big| \frac{1}{k} \sum_{j = 0}^{k-1} (\| \X_j \|_\infty - \E \| \X \|_\infty) \Big| > u }
        \le 2 e^{ - \frac{cku^2}{K^2 \| \SIGMA \|_\infty \log(p) } }.
    \end{align*}
    The first claim follows since $\E\|\X \|_\infty \lesssim K \sqrt{\| \SIGMA \|_\infty \log(2p)}$ by Lemma \ref{lem:linfExpectation}.

   For the second claim, we use that $\E\|\X\|_\infty \ge L \sqrt{\| \SIGMA \|_\infty}$ by Lemma \ref{lem:linfExpectation}. Combining this with the previous equation then leads to
    \begin{align*}
        \P{ \lambda_{k} < (1-t) L \sqrt{\| \SIGMA \|_\infty} }
        \le \P{|\lambda_{k} - \E \| \X \|_\infty | > t L \sqrt{\| \SIGMA \|_\infty} }
        \le 2 e^{ - \frac{ckL^2t^2}{K^2\log(p) } }
    \end{align*}
\end{proof}

Relying on Lemmas \ref{lem:linftyBiasEst} and \ref{lem:linfConcentration}, we can now derive the required bias control.

\begin{lemma}
\label{lem:biasControlUnified}
There exists an absolute constant $c$ such that the following holds.
Let $\X \in \mathbb R^p$ be a mean-zero, $K$-subgaussian vector with covariance matrix \sjoerd{$\SIGMA \in \R^{p\times p}$}.
Let $\X_0,...,\X_n \overset{\mathrm{i.i.d.}}{\sim} \X$ and let $\M \in [0,1]^{p\times p}$ be a fixed symmetric mask. Set $\tilde\lambda_k = \sqrt{\frac{4}{L^2 c_2} \log(k^2)} \lambda_k$, where $\lambda_k$ is defined in \eqref{eq:lambdaAdaptive}, and $c_2$ and $L$ are the constants from Lemmas \ref{lem:linftyBiasEst} and \ref{lem:linfConcentration} depending only on $K$.
Let $\tnorm{\cdot}$ be either the maximum, Frobenius, or operator norm. Then, for any $\theta\geq \max\{c, 2\log(n)\}$, we have with probability at least $1-2e^{-\theta}$
$$\tnorm{ \M \odot \Big( \sum_{k=1}^n \E_{k-1}( \tfrac{\tilde\lambda_k^2}{n} \Y_k\bar{\Y}_k^T) - \SIGMA \Big)}  \lesssim_{K} \theta \; \tnorm{\M} \| \SIGMA \|_\infty \frac{\log(2p) \log(n)}{n},$$
where we abbreviate the quantized samples in \eqref{eq:TwoBitSamplesAdaptive} by $\Y_k = \sign(\X_k + \Tau_k)$ and $\bar{\Y}_k =\sign(\X_k + \bar{\Tau}_k)$.
\end{lemma}

We extract the following technical observation from the proof of Lemma \ref{lem:biasControlUnified} since we are going to re-use it later.

\begin{lemma}
\label{lem:SumBounds}
There is an absolute constant $c>0$ such that the following holds. Let $\X \in \R^p$ be a $K$-subgaussian vector with $\E \X = \0$ and $\E(\X\X^\top) = \SIGMA \in \R^{p\times p}$. Let $\X_1,\dots,\X_n \overset{i.i.d.}{\sim} \X$ and let $\lambda_k$ be defined as in \eqref{eq:lambdaAdaptive}.
Set $\tilde\lambda_k = \sqrt{\frac{4}{L^2 c_2} \log(k^2)} \lambda_k$ for $k\geq 1$, where $c_2$  and $L$ are the constants from Lemmas \ref{lem:linftyBiasEst} and \ref{lem:linfConcentration} depending only on $K$. Then, for $\theta \ge c$ we have with probability at least $1-2e^{-\theta}$
\begin{equation}
\label{eqn:elemLambdaEstimate}
\tilde\lambda_k \lesssim \frac{K}{L} \sqrt{\theta\| \SIGMA \|_\infty \log(p)\log(k)}, \qquad \text{for all $1\leq k\le n$}.
\end{equation}
Moreover, for any $\theta\geq \max\{c, 2\log(n)\}$, we have with probability at least $1-2e^{-\theta}$ that
\begin{equation}
\label{eqn:sumLambdaEstimate1}
\sum_{k=1}^n (\tilde\lambda_k^2+\|\SIGMA\|_{\infty})e^{-c_2\tilde\lambda_k^2/\|\SIGMA\|_{\infty}}\lesssim \theta\frac{K^2}{L^2}\log(p)\log(n) \|\SIGMA\|_{\infty}
\end{equation}
and
\begin{equation}
\label{eqn:sumLambdaEstimate2}
\left(\sum_{k=1}^n \tilde\lambda_k^2(\tilde\lambda_k^2+\|\SIGMA\|_{\infty})e^{-c_2\tilde\lambda_k^2/\|\SIGMA\|_{\infty}}\right)^{1/2} \lesssim \theta\frac{K^2}{L^2}\log(p)\log(n) \|\SIGMA\|_{\infty}.
\end{equation}
\end{lemma}
\begin{proof}
Assuming that $\theta$ is sufficiently large to guarantee $e^{-c\theta} \le \frac{1}{2}$, for $c>0$ being the constant from Lemma~\ref{lem:linfConcentration}, one can combine \eqref{eqn:lambdakSubgaussian} in Lemma \ref{lem:linfConcentration} with a union bound over all $k\geq 1$, to obtain with probability at least $1-2e^{-c\theta}$ 
    \begin{align} \label{eq:LambdaUB}
        \lambda_k \lesssim K \sqrt{\theta\| \SIGMA \|_\infty \log(2p)},
    \end{align}
for all $k \ge 1$.
This proves \eqref{eqn:elemLambdaEstimate}.

Define now $k_\theta = \min\{\lceil \theta \frac{K^2}{c'L^2} \log(p) \rceil,n\}$, where $c'>0$ only depends on the constant $c$ from Lemma \ref{lem:linfConcentration}. If $k_\theta < n$, we find by \eqref{eqn:lambdakHPLowerBd} in Lemma \ref{lem:linfConcentration} and a union bound over all $k_{\theta} < k \leq n$ 
with probability at least 
\begin{align*}
    2 \sum_{k = k_\theta + 1}^n e^{ - \frac{ckL^2}{4 K^2\log(p) } } 
    \le 2n e^{-\theta} 
    \le 2 e^{-\frac{1}{2} \theta}
\end{align*}
that 
\begin{align}
\label{eq:LambdaKLB}
    \lambda_k \ge \frac{1}{2} L \sqrt{\| \SIGMA \|_\infty}, \qquad \text{for any $k_\theta < k \le n$}.
\end{align}
Note that we used in the union bound that $\theta - \log(n) \ge \frac{1}{2} \theta$ by assumption.
Upon these events we find
\begin{align*}
& \sum_{k=1}^n (\tilde\lambda_k^2+\|\SIGMA\|_{\infty})e^{-c_2\tilde\lambda_k^2/\|\SIGMA\|_{\infty}} \\
& \qquad = \sum_{\substack{1 \le k \le k_\theta: \\ \lambda_k^2\leq L^2\|\SIGMA\|_{\infty}}} (\tilde\lambda_k^2+\|\SIGMA\|_{\infty})e^{-c_2\tilde\lambda_k^2/\|\SIGMA\|_{\infty}}
+ \sum_{\substack{1 \le k \le k_\theta: \\ \lambda_k^2 > L^2\|\SIGMA\|_{\infty}}} (\tilde\lambda_k^2+\|\SIGMA\|_{\infty})e^{-c_2\tilde\lambda_k^2/\|\SIGMA\|_{\infty}} \\
&\qquad\qquad + \sum_{k=k_{\theta}+1}^n (\tilde\lambda_k^2+\|\SIGMA\|_{\infty})e^{-c_2\tilde\lambda_k^2/\|\SIGMA\|_{\infty}},
\end{align*}
where we recall that $\tilde\lambda_k = \sqrt{\frac{4}{L^2 c_2} \log(k^2)} \lambda_k$.
Observe that 
\begin{align*}
\sum_{\substack{1 \le k \le k_\theta: \\ \lambda_k^2\leq L^2\|\SIGMA\|_{\infty}}} (\tilde\lambda_k^2+\|\SIGMA\|_{\infty})e^{-c_2\tilde\lambda_k^2/\|\SIGMA\|_{\infty}} 
\lesssim k_{\theta}\log(k_{\theta}) \|\SIGMA\|_{\infty}
\leq \theta\frac{K^2}{L^2}\log(p)\log(n) \|\SIGMA\|_{\infty}
\end{align*}
and by \eqref{eqn:elemLambdaEstimate}
\begin{align*}
\sum_{\substack{1 \le k \le k_\theta: \\ \lambda_k^2 > L^2\|\SIGMA\|_{\infty}}} (\tilde\lambda_k^2+\|\SIGMA\|_{\infty})e^{-c_2\tilde\lambda_k^2/\|\SIGMA\|_{\infty}}
&\lesssim \sum_{k=1}^{k_{\theta}}(\tilde\lambda_k^2+\|\SIGMA\|_{\infty})e^{-\log(k^2)} \lesssim \theta \frac{K^2}{L^2} \|\SIGMA\|_{\infty}\log(p) \sum_{k=1}^{k_{\theta}}\frac{\log(k)}{k^2} \\
&\lesssim  \theta \frac{K^2}{L^2} \|\SIGMA\|_{\infty}\log(p).
\end{align*}
Similarly, if $k_{\theta}<n$ then \eqref{eqn:elemLambdaEstimate} and \eqref{eq:LambdaKLB} yield
\begin{align*}
\sum_{k=k_{\theta}+1}^n (\tilde\lambda_k^2+\|\SIGMA\|_{\infty})e^{-c_2\tilde\lambda_k^2/\|\SIGMA\|_{\infty}} 
&\lesssim  \sum_{k=k_{\theta}+1}^n \left(\theta \frac{ K^2}{L^2} \|\SIGMA\|_{\infty}\log(p)\log(k) + \|\SIGMA\|_{\infty}\right)e^{-\log(k^2)} \\
&\lesssim \theta \frac{K^2}{L^2} \|\SIGMA\|_{\infty}\log(p).
\end{align*}
Collecting these estimates we obtain \eqref{eqn:sumLambdaEstimate1}. The estimate in \eqref{eqn:sumLambdaEstimate2} is proved similarly.  
\end{proof}

\begin{proof}[Proof of Lemma \ref{lem:biasControlUnified}]
Let us first observe that 
\begin{equation}
\label{eqn:takeMout}
\tnorm{ \M \odot \Big( \sum_{k=1}^n \E_{k-1}( \tfrac{\tilde\lambda_k^2}{n} \Y_k\bar{\Y}_k^T) - \SIGMA \Big)}\leq \tnorm{\M} \pnorm{\sum_{k=1}^n \E_{k-1}( \tfrac{\tilde\lambda_k^2}{n} \Y_k\bar{\Y}_k^T) - \SIGMA}{\infty}.
\end{equation}
This is clear for the maximum and Frobenius norms. For the remaining case $\tnorm{\cdot} = \| \cdot \|$, observe that $\M$ has only \sjoerd{nonnegative} entries and hence
	\begin{align}
	\label{eqn:normEstPos}
	\pnorm{ \M \odot \Big( \sum_{k=1}^n \E_{k-1}( \tfrac{\tilde\lambda_k^2}{n} \Y_k\bar{\Y}_k^T) - \SIGMA \Big)}{} \nonumber
	&=  \sup_{\v,\w \in \R^p \ : \ \|\v\|_2, \|\w\|_2 \leq 1} \abs{\sum_{i,j=1}^p M_{i,j} \Big( \sum_{k=1}^n \E_{k-1}( \tfrac{\tilde\lambda_k^2}{n} \Y_k\bar{\Y}_k^T) - \SIGMA \Big)_{i,j} v_i w_j }{} \nonumber\\
	&
	\le  \sup_{\v,\w \in \R^p \ : \ \|\v\|_2, \|\w\|_2 \leq 1} \sum_{i,j=1}^p M_{i,j} \abs{ \Big( \sum_{k=1}^n \E_{k-1}( \tfrac{\tilde\lambda_k^2}{n} \Y_k\bar{\Y}_k^T) - \SIGMA \Big)_{i,j} }{} \abs{v_i w_j}{}  \nonumber\\
	&\le \pnorm{\M}{} \pnorm{\sum_{k=1}^n \E_{k-1}( \tfrac{\tilde\lambda_k^2}{n} \Y_k\bar{\Y}_k^T) - \SIGMA}{\infty}.
\end{align}
To complete the proof, we apply Lemma~\ref{lem:linftyBiasEst} conditionally and use Lemma~\ref{lem:SumBounds} to obtain with probability at least $1-2e^{-c\theta}$
\begin{align*}
\pnorm{\sum_{k=1}^n \E_{k-1}( \tilde\lambda_k^2 \Y_k\bar{\Y}_k^T) - \SIGMA}{\infty}
	& \leq \sum_{k=1}^n \pnorm{ \E_{k-1}( \tilde\lambda_k^2 \Y_k\bar{\Y}_k^T) - \SIGMA}{\infty}
	\nonumber \\
	&
	\sjoerd{\lesssim_K} \sum_{k=1}^n (\tilde\lambda_k^2+\|\SIGMA\|_{\infty})e^{-c_2\tilde\lambda_k^2/\|\SIGMA\|_{\infty}} \\
	& \lesssim_{K} \theta\log(p) \log(n)\| \SIGMA \|_\infty,
\end{align*}
where in the final step we applied \eqref{eqn:sumLambdaEstimate1} and used that $L$ in Lemma \ref{lem:SumBounds} depends solely on $K$. 
\end{proof}

\subsection{Proof of Theorem~\ref{thm:FrobeniusDitheredMask}}
\label{sec:ProofOfFrobeniusDitheredMask}


With Lemma~\ref{lem:biasControlUnified} at hand, it remains to estimate the martingale difference term in \eqref{eqn:splitExpDithered}, i.e., the first term on the right hand side of that equation. The martingale differences in question are not uniformly bounded, so that we cannot use the Azuma-Hoeffding inequality to estimate this term. Instead, we will use the following classical extension of Rosenthal's inequality \cite{Ros70}, due to Burkholder \cite{Bur73}. The stated version with the constants $\sqrt{q}$ and $q$ can be obtained by combining \cite{NaP78} (who proved this result for independent summands) with a decoupling argument \cite{Hit94}, see \cite[Remark 3.6]{DiY19} for a more detailed discussion.

\begin{theorem}
\label{thm:BRreal}
There exists a constant $c>0$ such that for any real-valued martingale difference sequence $(\xi_k)_{k=1}^n$ 
and any $2\leq q<\infty$,
\begin{equation}
\label{eqn:BRreal}
\left(\E\left|\sum_{k=1}^n \xi_k\right|^q\right)^{\frac{1}{q}} \leq c\sqrt{q}\left(\E\left(\sum_{k=1}^n \E_{k-1}|\xi_k|^2\right)^{\frac{q}{2}}\right)^{\frac{1}{q}}+ cq\left(\sum_{k=1}^n\E |\xi_k|^q\right)^{\frac{1}{q}}.
\end{equation}
\end{theorem}
To prove Theorem \ref{thm:FrobeniusDitheredMask}, we will use Theorem~\ref{thm:BRreal} in combination with the following well-known observation, which formalizes that one can pass between tail bounds and $L^q$-bounds at the expense of additional (absolute) constant factors. We provide a proof to keep our exposition self-contained.
\begin{lemma}
\label{lem:LqtoTailBound}
    Assume $0\leq q_0<q_1\leq \infty$ and let $\xi$ be a random variable. 
    \begin{enumerate}
        \item[(i)] Let $J \in \mathbb N$, $p_j>0$ and $a_j\geq 0$ for $j=1,\ldots,J$, and suppose that
        \begin{align*}
            (\mathbb{E}|\xi|^q)^{1/q}\leq \sum_{j=1}^J a_j q^{p_j},
        \end{align*}
        for all $q_0\leq q\leq q_1$. Then, 
        \begin{align*}
            \P{|\xi|\geq e \cdot \sum_{j=1}^J a_j t^{p_j}} \leq e^{-t}
        \end{align*}
        for all $q_0\leq t\leq q_1$. 
        \item[(ii)] If
        \begin{align*}
            \P{|\xi|\geq  a t^{p}} \leq e^{-t},
        \end{align*}
        for fixed $a>0$, $p\geq \tfrac{1}{2}$, and any $t \ge q_0$, then
        \begin{align*}
            (\mathbb{E}|\xi|^q)^{1/q} \lesssim aq_0^p + p^{p} a q^p,
        \end{align*}
        for any $q \ge 1$.
    \end{enumerate}
\end{lemma}

\begin{proof}
		To show $(i)$, 
  we use Markov's inequality to find
		\begin{align*}
		    \johannes{\P{|\xi|\geq e\sum_{j=1}^J a_j t^{p_j}}}{}
    		\leq \frac{\mathbb{E}|\xi|^t}{(e\sum_{j=1}^J a_j t^{p_j})^t}\leq e^{-t}.
		\end{align*}

        To see $(ii)$, note that by assumption $\P{|\xi|\geq  u \}} \leq e^{-\frac{u^{1/p}}{a^{1/p}}}$ for all $u\geq a q_0^p$. Hence, for any $q \ge 1$
        \begin{align*}
            \mathbb{E}|\xi|^q 
            &= \int_0^{aq_0^p} q u^{q-1} \P{|\xi|\geq  u } du
            + \int_{aq_0^p}^\infty q u^{q-1} \P{|\xi|\geq  u } du 
            \le a^q q_0^{qp} + q \cdot \int_{aq_0^p}^\infty u^{q-1} \cdot e^{-\frac{u^{1/p}}{a^{1/p}}} du \\
            &\leq a^q q_0^{qp} + q \cdot \int_{0}^\infty (a t^p)^{q-1} \cdot e^{-t} \cdot (apt^{p-1}) dt 
            = a^q q_0^{qp} + pq \cdot a^q \cdot \Gamma(pq) \\
            &\le a^q q_0^{qp} + \sqrt{2\pi} a^q \cdot (pq)^{pq+\frac{1}{2}} \cdot e^{-pq},
        \end{align*}
        where we used integration by parts and Stirling's approximation of the Gamma function. Taking the $q$-th root on both sides of the inequality now yields $(ii)$.
\end{proof}

Finally, we have all the tools we need to prove Theorem \ref{thm:FrobeniusDitheredMask}. 
%

\begin{proof}[Proof of Theorem~\ref{thm:FrobeniusDitheredMask}]
We only need to show the maximum norm bound. The Frobenius norm bound then follows from the fact that $\| \M \odot \A \|_F \le \| \M \|_F \| \A  \|_\infty$, for any $\A \in \R^{p\times p}$.

Recall \eqref{eqn:splitExpDithered} and abbreviate $\Y_k = \sign(\X_k + \Tau_k)$ and $\bar{\Y}_k = \sign(\X_k + \bar{\Tau}_k)$. Set $\tilde\lambda_k = \sqrt{\frac{4}{L^2 c_2} \log(k^2)} \lambda_k$ for $k\geq 1$, where $c_2$ and $L$ are the constants from Lemmas \ref{lem:linftyBiasEst} and \ref{lem:linfConcentration}, respectively. Applying Lemma~\ref{lem:biasControlUnified} to the second term on the right hand side of \eqref{eqn:splitExpDithered} yields
\begin{align}
\label{eq:SecondTermBoundLinfty}
    \left\| \M \odot \Big( \sum_{k=1}^n \E_{k-1}( \tfrac{\tilde\lambda_k^2}{n} \Y_k\bar{\Y}_k^T) - \SIGMA \Big) \right\|_\infty  \lesssim_{K} \theta \| \M \|_\infty \| \SIGMA \|_\infty \frac{\log(2p) \log(n)}{n}.
\end{align} 
To complete the proof, we estimate the first term on the right hand side of \eqref{eqn:splitExpDithered}.
First note that for $q \geq \log(m)$ and random variables $Z_1,\dots,Z_m$
\begin{align}
\label{eq:Lq_maxBound}
    \left( \E \max_{i \in [m]} |Z_i|^q \right)^{\frac{1}{q}}
    \le \left( \sum_{i=1}^m \E |Z_i|^q \right)^{\frac{1}{q}}
    \le m^{\frac{1}{\log(m)}} \max_{i \in [m]} \left( \E |Z_i|^q \right)^{\frac{1}{q}} 
    \lesssim \max_{i \in [m]} \left( \E |Z_i|^q \right)^{\frac{1}{q}}
\end{align}
so that for any $q \ge \log(p^2)$
\begin{align}
\label{eq:FirstTermBoundLinfty}
    \left(\E\left\|\SIGMA'_n - \sum_{k=1}^n \E_{k-1}( \tfrac{\tilde\lambda_k^2}{n} \Y_k\bar{\Y}_k^T)\right\|_{\infty}^q\right)^{1/q} 
    & \lesssim \frac{1}{n} \max_{1\leq i,j\leq p}\left(\E\left(\sum_{k=1}^n \tilde\lambda_k^2 \Big( \Y_k\bar{\Y}_k^T - \E_{k-1}( \Y_k\bar{\Y}_k^T) \Big)_{ij}\right)^q\right)^{1/q}.        
\end{align}
We aim to apply Theorem~\ref{thm:BRreal} to the right-hand side of \eqref{eq:FirstTermBoundLinfty}.
Fix $i,j \in [p]$ and define the sequence
$$\xi_k := \tilde\lambda_k^2 \Big( (Y_k)_i (\bar{Y}_k)_j - \E_{k-1}( (Y_k)_i (\bar{Y}_k)_j) \Big).$$
We need to check the assumptions of Theorem \ref{thm:BRreal} and to estimate the quantities appearing on the right-hand side of \eqref{eqn:BRreal}.
Clearly, $(\xi_k)$ forms a real-valued martingale difference sequence with respect to the filtration defined by \eqref{eqn:filtrationDef}, i.e.,
$$\mathcal{F}_k = \sigma(\X_0,\X_1,\ldots,\X_k,\Tau_1,\ldots,\Tau_k,\bar{\Tau}_1, \ldots,\bar{\Tau}_k), \qquad k=0,1,\ldots,n.$$
Further observe that $\tilde\lambda_k$ is $\mathcal{F}_{k-1}$-measurable and hence
\begin{align*}
    \E_{k-1}|\xi_k|^2 = \tilde\lambda_k^4 \E_{k-1}\Big[\Big( (Y_k)_i (\bar{Y}_k)_j - \E_{k-1}( (Y_k)_i (\bar{Y}_k)_j) \Big)^2\Big]\leq 4 \tilde\lambda_k^4
\end{align*}
since $(Y_k)_i,(Y_k)_j \in \{-1,1\}$. Hence, by combining \eqref{eqn:elemLambdaEstimate} in Lemma \ref{lem:SumBounds} with Lemma \ref{lem:LqtoTailBound} we obtain for any $q \ge 1$ that
\begin{align*}
    \left\|\left(\sum_{k=1}^n \E_{k-1}|\xi_k|^2\right)^{\frac{1}{2}}\right\|_{L^q} 
    &\leq 2 \left(\sum_{k=1}^n \|\tilde\lambda_k^4\|_{L^{\frac{q}{2}}}\right)^{1/2} 
    = 2\left(\sum_{k=1}^n \|\tilde\lambda_k\|_{L^{2q}}^4\right)^{1/2} \\
    &\lesssim_{K}  \left(\sum_{k=1}^n q^2\log^2(p)\log^2(k)\|\SIGMA\|_{\infty}^2\right)^{1/2} \\
    &\leq \sqrt{n} \cdot q\log(p)\log(n)\|\SIGMA\|_{\infty}.
\end{align*} 
Along similar lines, we get for any $q \ge \log(n)$ that
\begin{align*}
    \left(\sum_{k=1}^n \|\xi_k\|_{L^q}^q \right)^{\frac{1}{q}}
    &\leq 2\left(\sum_{k=1}^n\|\tilde\lambda_k^2\|_{L^q}^q\right)^{\frac{1}{q}} \lesssim_K n^{1/q} \cdot q\log(p)\log(n)\|\SIGMA\|_{\infty} \\
    &\lesssim q\log(p)\log(n)\|\SIGMA\|_{\infty}.
\end{align*}
Applying Theorem~\ref{thm:BRreal} now to the right-hand side of \eqref{eq:FirstTermBoundLinfty} yields 
 \begin{align*}
        \left(\E\left\|\SIGMA'_n - \sum_{k=1}^n \E_{k-1}( \tfrac{\tilde\lambda_k^2}{n} \Y_k\bar{\Y}_k^T)\right\|_{\infty}^q\right)^{1/q} \lesssim_K \frac{q^{3/2}\log(p)\log(n)\|\SIGMA\|_{\infty}}{\sqrt{n}} + \frac{q^2\log(p)\log(n)\|\SIGMA\|_{\infty}}{n},
\end{align*}
for any $q \ge \max\{ 2\log(q),\log(n) \}$. Translating this into a tail bound via Lemma \ref{lem:LqtoTailBound} and inserting it together with \eqref{eq:SecondTermBoundLinfty} into the right-hand side of \eqref{eqn:splitExpDithered} concludes the proof. 
\end{proof}

\subsection{Proof of Theorem \ref{thm:OperatorDitheredMask}}
\label{sec:ProofOfOperatorDitheredMask}

The proof follows the lines of the proof of Theorem~\ref{thm:FrobeniusDitheredMask}. We again use Lemmas~\ref{lem:biasControlUnified} and \ref{lem:SumBounds}. The main difference is that we use the Burkholder-Rosenthal for matrix martingales (Theorem \ref{thm:matrixBRintro}) instead of Theorem \ref{thm:BRreal}. We will additionally need the following observation, which is a simple consequence of Schur's product theorem \cite[Eq. (3.7.11)]{johnson1990matrix}.

\begin{lemma}
\label{lem:HadamardOperatorNorm}
	Let $\A,\B \in \R^{p\times p}$, let $\A$ be symmetric and define $\C = (\A^T\A)^\frac{1}{2}$. Then
	\begin{align*}
	\pnorm{\A \odot \B}{} \le \round{\max_{i \in [p]} C_{ii}} \pnorm{\B}{} = \pnorm{\A}{1\rightarrow 2} \pnorm{\B}{}.
	\end{align*}{}
	If $\A$ is in addition positive semidefinite, then
	\begin{align*}
	\pnorm{\A \odot \B}{} \le \round{\max_{i \in [p]} A_{ii} } \pnorm{\B}{}.
	\end{align*}
\end{lemma}{}

%

\begin{proof}[Proof of Theorem~\ref{thm:OperatorDitheredMask}]

    Recall that $\Y_k = \sign(\X_k + \Tau_k)$, $\bar{\Y}_k = \sign(\X_k + \bar{\Tau}_k)$, and set $\tilde\lambda_k = \sqrt{\frac{4}{L^2 c_2} \log(k^2)} \lambda_k$ for $k\geq 1$, where $c_2$ and $L$ are the constants from Lemmas~\ref{lem:linftyBiasEst} and \ref{lem:linfConcentration}, respectively. 
We again use the split in \eqref{eqn:splitExpDithered}. Applying Lemma~\ref{lem:biasControlUnified} to the second term on the right hand side of \eqref{eqn:splitExpDithered} yields
    \begin{align}
    \label{eq:SecondTermBoundLoperator}
        \left\| \M \odot \Big( \sum_{k=1}^n \E_{k-1}( \tfrac{\tilde\lambda_k^2}{n} \Y_k\bar{\Y}_k^T) - \SIGMA \Big) \right\|  \lesssim_{K} \theta \| \M \| \| \SIGMA \|_\infty \frac{\log(2p) \log(n)}{n}.
    \end{align} 
    To complete the proof, we will estimate the first term on the right hand side of \eqref{eqn:splitExpDithered} using Theorem~\ref{thm:matrixBRintro}. To this end, define
    \begin{align*}
        {\XI}_k 
        = \frac{\tilde\lambda_k^2}{n} \M\odot \round{ \Y_k\bar{\Y}_k^T - \E_{k-1}\round{\Y_k\bar{\Y}_k^T} }{}
        \quad \text{ \sjoerd{so} that } \quad
        \M \odot \Big( \SIGMA'_n - \sum_{k=1}^n \E_{k-1}( \tfrac{\tilde\lambda_k^2}{n} \Y_k\bar{\Y}_k^T) \Big) = \sum_{k=1}^n \XI_k.
    \end{align*}
Since $\E\|\XI_k\| < \infty$ and $\E_{k-1} \XI_k = \0$ by construction, the sequence $\XI_1,\dots,\XI_n$ forms a matrix martingale difference sequence with respect to the filtration \eqref{eqn:filtrationDef} and we can hence apply Theorem~\ref{thm:matrixBRintro}. Let us now estimate the quantities appearing in the bounds of Theorem \ref{thm:matrixBRintro}. For any $1\leq k\leq n$,
    \begin{equation*}
        \|{\XI}_k\| = \frac{\tilde\lambda_k^2}{n} \| \M \odot \Y_k\bar{\Y}_k^T - \E_{k-1}\round{\M \odot \Y_k\bar{\Y}_k^T} \| 
        \leq \frac{2\tilde\lambda_k^2 \|\M\|}{n}
    \end{equation*}
    since $\M \odot \Y_k \bar{\Y}_k^T = \diag(\Y_k) \; \M  \; \diag(\bar{\Y}_k)$ and $\pnorm{\diag(\Y_k)}{} = 1$. Hence,
    \begin{equation}
    \label{eqn:XikMaxNormEstimateLq}
        \max_{1\leq k\leq n} (\E\|{\XI}_k\|^q)^{1/q} \lesssim \frac{\|\M\|}{n}\max_{1\leq k\leq n} \| \tilde\lambda_k^2 \|_{L^q}\lesssim_K \frac{\|\M\|}{n}q\|\SIGMA\|_{\infty}\log(p)\log(n)
    \end{equation} 
    where we used in the last inequality a moment bound for $\tilde\lambda_k^2$ derived from \eqref{eqn:elemLambdaEstimate} via Lemma \ref{lem:LqtoTailBound}. 
    Next, using that $(\Y_k,\bar{\Y}_k)$ and $(\bar{\Y}_k,\Y_k)$ are identically distributed, we get 
    \begin{align*}
    & \Big\|\Big(\sum_{k=1}^n \mathbb{E}_{k-1}(\XI_k^T\XI_k) \Big)^{1/2}\Big\| \\
    &\qquad = \pnorm{ \sum_{k=1}^n \frac{\tilde\lambda_k^4}{n^2} \left( \mathbb{E}_{k-1}[(\M \odot \Y_k\bar{\Y}_k^T)^T(\M \odot \Y_k\bar{\Y}_k^T)] - \Big(\M \odot \johannes{\E_{k-1}\round{ \Y_k\bar{\Y}_k^T }}{} \Big)^T \Big(\M \odot \johannes{\E_{k-1}\round{ \Y_k\bar{\Y}_k^T }}{} \Big) \right) }{}^{1/2}\\
    &\qquad= \pnorm{ \sum_{k=1}^n \frac{\tilde\lambda_k^4}{n^2} \left( \M^2\odot\johannes{\E_{k-1}(\bar{\Y}_k\bar{\Y}_k^T)}{} - \Big( \M \odot \johannes{\E_{k-1}\round{ \Y_k\bar{\Y}_k^T }}{} \Big)^T \Big( \M \odot \johannes{\E_{k-1}\round{ \Y_k\bar{\Y}_k^T }}{} \Big) \right) }{}^{1/2}\\
    &\qquad= \pnorm{ \sum_{k=1}^n \frac{\tilde\lambda_k^4}{n^2} \left( \M^2\odot \johannes{\E_{k-1}(\Y_k\Y_k^T)}{} - \Big( \M \odot \johannes{\E_{k-1}\round{ \bar{\Y}_k\Y_k^T }}{} \Big)^T \Big( \M \odot \johannes{\E_{k-1}\round{ \bar{\Y}_k\Y_k^T }}{} \Big) \right) }{}^{1/2},
    \end{align*}
    where we used in the second equality that by $\diag(\Y_k)^2 = \id$
    \begin{align}
    \label{eq:Msquared}
        (\M\odot \Y_k\bar{\Y}_k^T)^T (\M\odot \Y_k\bar{\Y}_k^T)
       = \diag(\bar{\Y}_k)\;\M\; \diag(\Y_k)\;\diag(\Y_k)\;\M\;\diag(\bar{\Y}_k)
       =\M^2 \odot \bar{\Y}_k\bar{\Y}_k^T.
    \end{align}
    Interchanging the roles of $\Y_k$ and $\bar{\Y}_k$ yields
    \begin{align*}
    & \Big\|\Big(\sum_{k=1}^n \mathbb{E}_{k-1}(\XI_k\XI_k^T) \Big)^{1/2}\Big\| \\
    & \qquad = \pnorm{ \sum_{k=1}^n \frac{\tilde\lambda_k^4}{n^2} \left( \M^2\odot \johannes{\E_{k-1}(\Y_k\Y_k^T)}{} - \Big( \M \odot \johannes{\E_{k-1}\round{ \bar{\Y}_k\Y_k^T }}{} \Big)^T \Big( \M \odot \johannes{\E_{k-1}\round{ \bar{\Y}_k\Y_k^T }}{} \Big) \right) }{}^{1/2}.
    \end{align*}
    Since by \eqref{eq:Kadison}
    $$\Big( \M \odot \johannes{\E_{k-1}\round{ \bar{\Y}_k\Y_k^T }}{} \Big)^T \Big( \M \odot \johannes{\E_{k-1}\round{ \bar{\Y}_k\Y_k^T }}{} \Big) 
    \preceq \E_{k-1} \Big( (\M \odot \bar{\Y}_k\Y_k^T )^T ( \M \odot \bar{\Y}_k\Y_k^T ) \Big) $$ 
    we find
    \begin{align*}
    & \max\left\{ \Big\|\Big(\sum_{k=1}^n \mathbb{E}_{k-1}(\XI_k^T\XI_k) \Big)^{1/2}\Big\|, \Big\|\Big(\sum_{k=1}^n \mathbb{E}_{k-1}(\XI_k\XI_k^T) \Big)^{1/2}\Big\| \right\} \\
    &\quad \leq \left( \sum_{k=1}^n \frac{2 \tilde\lambda_k^4}{n^2} \pnorm{ \M^2\odot \E_{k-1}(\Y_k\Y_k^T)}{} \right)^{1/2}\\
    & \quad \leq \left( \sum_{k=1}^n \frac{2 \tilde\lambda_k^2}{n^2} \pnorm{ \M^2\odot \left(\tilde\lambda_k^2\johannes{\E_{k-1}\round{\Y_k\Y_k^T}}{} - (\SIGMA-\diag(\SIGMA)+\tilde\lambda_k^2 \id) \right) }{} \right. \\
    & \quad \qquad \qquad \qquad \qquad \qquad \qquad \qquad \qquad + \left. \sum_{k=1}^n \frac{2 \tilde\lambda_k^2}{n^2} \pnorm{ \M^2\odot (\SIGMA-\diag(\SIGMA)+\tilde\lambda_k^2 \id) }{} \right)^{1/2},
    \end{align*}
    where we re-used \eqref{eq:Msquared} in the first inequality.
    By the same reasoning as in \eqref{eqn:normEstPos} together with Lemma~\ref{lem:linftyBiasEst} 
    \begin{align*}
    &\pnorm{ \M^2\odot \left(\tilde\lambda_k^2\johannes{\E_{k-1}\round{\Y_k\Y_k^T}}{} - (\SIGMA-\diag(\SIGMA)+\tilde\lambda_k^2 \id) \right) }{} \\
    &\quad \leq \|\M^2\| \pnorm{\tilde\lambda_k^2\johannes{\E_{k-1}\round{\Y_k\Y_k^T}}{} - (\SIGMA-\diag(\SIGMA)+\tilde\lambda_k^2 \id)}{\infty} 
     \sjoerd{\lesssim_K} \|\M\|^2 (\tilde\lambda_k^2+\|\SIGMA\|_{\infty})e^{-c_2\tilde\lambda_k^2/\|\SIGMA\|_{\infty}} 
    \end{align*}
    Moreover, by Lemma~\ref{lem:HadamardOperatorNorm}, we find
	\begin{align*}
	\|\M^2 \odot (\SIGMA-\diag(\SIGMA)+\tilde\lambda_k^2 \id) \| 
	&\leq \|\M^2\odot\SIGMA\|+ \|\M^2\odot \id \odot\SIGMA\| + \tilde\lambda_k^2 \|\M^2\odot \id \| \\
	&\leq \|\M\|_{1\to 2}^2(2\|\SIGMA\|+\tilde\lambda_k^2).
	\end{align*}
    Thus, we obtain
	\begin{align}
	\label{eqn:XikSquareFunEstimate}
	& \max\left\{ \Big\|\Big(\sum_{k=1}^n \mathbb{E}_{k-1}(\XI_k^T\XI_k) \Big)^{1/2}\Big\|, \Big\|\Big(\sum_{k=1}^n \mathbb{E}_{k-1}(\XI_k\XI_k^T) \Big)^{1/2}\Big\| \right\} \nonumber\\
		& \quad \lesssim \frac{\|\M\|}{n}\left( \sum_{k=1}^n\tilde\lambda_k^2(\tilde\lambda_k^2+\|\SIGMA\|_{\infty})e^{-c_2\tilde\lambda_k^2/\|\SIGMA\|_{\infty}} \right)^{1/2} + \frac{\|\M\|_{1\to 2}}{n} \left( \sum_{k=1}^n \tilde\lambda_k^2(\|\SIGMA\|+\tilde\lambda_k^2) \right)^{1/2}, 	
    \end{align}
    By translating \eqref{eqn:sumLambdaEstimate1} in Lemma \ref{lem:SumBounds} into a moment bound via Lemma \ref{lem:LqtoTailBound}, we find that
    \begin{align*}
        \left\|\left( \sum_{k=1}^n\tilde\lambda_k^2(\tilde\lambda_k^2+\|\SIGMA\|_{\infty})e^{-c_2\tilde\lambda_k^2/\|\SIGMA\|_{\infty}} \right)^{1/2}\right\|_{L^q} \lesssim_{K} q\log(p)\log(n) \|\SIGMA\|_{\infty}.
    \end{align*}
    Moreover, by again translating \eqref{eqn:elemLambdaEstimate} in Lemma \ref{lem:SumBounds} into a moment bound via Lemma \ref{lem:LqtoTailBound}, we further get that
    \begin{align*}
        \left\|\left( \sum_{k=1}^n \tilde\lambda_k^2(\|\SIGMA\|+\tilde\lambda_k^2) \right)^{1/2}\right\|_{L^q} & \leq \left(\sum_{k=1}^n \|\tilde\lambda_k^2\|_{L^{\frac{q}{2}}}\|\SIGMA\|+\|\tilde\lambda_k^4\|_{L^{\frac{q}{2}}} \right)^{1/2}\\
        & \lesssim_{K} \left(\sum_{k=1}^n q\log(p)\log(k)\|\SIGMA\|_{\infty}\|\SIGMA\|+q^2\log^2(p)\log^2(k)\|\SIGMA\|_{\infty}^2\right)^{1/2} \\
        & = \sqrt{n}\left(q\log(p)\log(n)\|\SIGMA\|_{\infty}\|\SIGMA\| +q^2\log^2(p)\log^2(n)\|\SIGMA\|_{\infty}^2\right)^{1/2}\\
        & \lesssim \sqrt{n}q\log(p)\log(n)\|\SIGMA\|_{\infty}^{1/2}\|\SIGMA\|^{1/2}
    \end{align*}
    Thus, taking $L^q$-norms in \eqref{eqn:XikSquareFunEstimate} we find
    	\begin{align}
    	\label{eqn:XikSquareFunEstimateLq}
    	& \max\left\{\left(\mathbb{E}\left\|\left(\sum_{k=1}^n \mathbb{E}_{k-1}(\XI_k^T\XI_k) \right)^{1/2}\right\|^q\right)^{1/q}, \left(\mathbb{E}\left\|\left(\sum_{k=1}^n \mathbb{E}_{k-1}(\XI_k\XI_k^T) \right)^{1/2}\right\|^q\right)^{1/q} \right\} \nonumber \\
    	& \quad \lesssim_K q\log(p)\log(n)\left(\frac{\|\M\|}{n} \|\SIGMA\|_{\infty} + \frac{\|\M\|_{1\to 2}}{\sqrt{n}}\|\SIGMA\|_{\infty}^{1/2}\|\SIGMA\|^{1/2}\right).
    	\end{align}
    By using Theorem~\ref{thm:matrixBRintro} and inserting \eqref{eqn:XikMaxNormEstimateLq} and \eqref{eqn:XikSquareFunEstimateLq} we finally get for any $q \ge \max\{\log(p),\log(n)\}$ that 
    \begin{align*}
    & \left(\mathbb{E}\left\|\M \odot \left( \SIGMA'_n - \sum_{k=1}^n \E_{k-1}( \tfrac{\tilde\lambda_k^2}{n} \Y_k\bar{\Y}_k^T) \right)\right\|^q\right)^{1/q} \\
    & \qquad \lesssim_K q^2\log(p)\log(n)\left(\sqrt{q}\frac{\|\M\|_{1\to 2}}{\sqrt{n}}\|\SIGMA\|_{\infty}^{1/2}\|\SIGMA\|^{1/2}+q\frac{\|\M\|}{n} \|\SIGMA\|_{\infty}\right).
    \end{align*}
    The result now immediately follows by applying Lemma~\ref{lem:LqtoTailBound} to the previous moment bound and inserting the resulting tail estimate with \eqref{eq:SecondTermBoundLoperator} into the right-hand side of \eqref{eqn:splitExpDithered}. 
\end{proof}

\subsection{Proof of Theorem~\ref{thm:matrixBRintro}}
\label{sec:Burkholder}

We will follow an argument from \cite{DiY19}, where Burkholder-Rosenthal inequalities were derived for $L^p$-valued martingale difference sequences. Our starting point is the following $L^q$-version of the matrix Bernstein inequality \cite{tropp2012user}, see \cite[Theorem 6.2]{Dir14}.
\begin{theorem}
	\label{thm:sumsRMs} 
	Let $2\leq q<\infty$. If $({\THETA}_k)_{k=1}^n$ is a sequence of independent, mean-zero random matrices in $\R^{p_1\times p_2}$, then
	\begin{align*}
	\johannes{\Big(\E \Big\|\sum_{k=1}^n {\THETA}_k\Big\|^q \Big)^{1/q}} & \leq C_{q,p}^{\operatorname{(MR)}} \max\Big\{\Big\|\Big(\sum_{k=1}^n \mathbb{E}(\THETA_k^T\THETA_k)\Big)^{1/2}\Big\|, \Big\|\Big(\sum_{k=1}^n \mathbb{E}(\THETA_k\THETA_k^T)\Big)^{1/2}\Big\|,\\
	& \qquad \qquad\qquad\qquad\qquad\qquad \qquad\ \ \ \ \ 2C_{\frac{q}{2},p}^{\operatorname{(MR)}} \johannes{\Big(\E \max_{1\leq k\leq n} \|{\THETA}_k\|^q\Big)^{1/q}\Big\},}
	\end{align*}
	where $p=\max\{p_1,p_2\}$ and
	$C_{q,p}^{\operatorname{(MR)}}\leq  2^{\frac{3}{2}} e (1 + \sqrt{2}) \sqrt{\max\{q,\log p\}}$.  Moreover, the reverse inequality holds with universal constants. 
\end{theorem}
We will derive Theorem~\ref{thm:matrixBRintro} by combining Theorem~\ref{thm:sumsRMs} with a general decoupling technique from \cite{KwW92}. Let
$(\Om,\cF,\bP)$ be a complete probability space, let
$(\cF_i)_{i\geq 0}$ be a filtration, and let $X$ be a Banach space. We denote the Borel $\sigma$-algebra on $X$ by $\mathcal{B}(X)$.
Two $(\cF_i)_{i\geq 1}$-adapted sequences $(\xi_i)_{i\geq 1}$ and
$(\theta_i)_{i\geq 1}$ of $X$-valued random variables are called
\emph{tangent} if for every $i\geq 1$ and $A \in \mathcal{B}(X)$
\begin{equation}
\label{eqn:tangent}
\bP(\xi_i \in A|\cF_{i-1}) = \bP(\theta_i \in A|\cF_{i-1}).
\end{equation}
An $(\cF_i)_{i\geq 1}$-adapted sequence $(\theta_i)_{i\geq 1}$ of $X$-valued random variables is said to satisfy \emph{condition (CI)} if, firstly, there exists a
sub-$\sigma$-algebra
$$
\cG\subset\cF_{\infty}=\sigma(\cup_{i\geq 0}\cF_i)
$$
such that for every $i\geq 1$ and $A \in \mathcal{B}(X)$,
\begin{equation}
\label{eqn:CI}
\bP(\theta_i \in A|\cF_{i-1}) = \bP(\theta_i \in A|\cG)
\end{equation}
and, secondly, $(\theta_i)_{i\geq 1}$ consists of $\cG$-independent random variables, i.e.\ for all $n\geq 1$ and $A_1,\ldots,A_n \in \mathcal{B}(X)$,
\begin{align}
\label{eq:G_independence}
    \E(\mathbf 1_{\theta_1 \in A_1}\cdot\ldots\cdot \mathbf 1_{\theta_n \in A_n}|\cG) = \E(\mathbf 1_{\theta_1 \in A_1}|\cG)\cdot\ldots\cdot\E(\mathbf 1_{\theta_n \in A_n}|\cG).
\end{align}
It is shown in \cite{KwW92} that for every $(\cF_i)_{i\geq
1}$-adapted sequence $(\xi_i)_{i\geq 1}$ there exists an
$(\cF_i)_{i\geq 1}$-adapted sequence $(\theta_i)_{i\geq 1}$ on a
possibly enlarged probability space which is tangent to
$(\xi_i)_{i\geq 1}$ and satisfies condition (CI). This sequence is
called a \emph{decoupled tangent sequence} for $(\xi_i)_{i\geq 1}$
and is unique in law.\par 
To prove Theorem \ref{thm:matrixBRintro}, we will set $(\xi_i)_{i\ge 1} = (\XI_i)_{i\ge 1}$ and apply
Theorem~\ref{thm:sumsRMs} conditionally to its decoupled
tangent sequence $(\theta_i)_{i\geq 1}$. For this approach to work, we
will need to relate various norms of $(\xi_i)_{i\geq 1}$ and
$(\theta_i)_{i\geq 1}$. One of these estimates can be formulated as a
Banach space property. A
Banach space $X$ satisfies \emph{the $q$-decoupling property}
for some $1\leq q<\infty$ if there is a constant $C_{q,X}$ such that for
any complete probability space $(\Om,\cF,\bP)$, any filtration
$(\cF_i)_{i\geq 0}$, and any $(\cF_{i})_{i\geq 1}$-adapted
sequence $(\xi_i)_{i\geq 1}$ in $L^q(\Om,X)$,
\begin{equation}
\label{eqn:decouple}
\Big(\E\Big\|\sum_{i=1}^n \xi_i\Big\|_X^q\Big)^{1/q} \leq C_{q,X} \Big(\E\Big\|\sum_{i=1}^n \theta_i\Big\|_X^q\Big)^{1/q},
\end{equation}
for all $n\geq 1$, where $(\theta_i)_{i\geq 1}$ is any decoupled
tangent sequence of $(\xi_i)_{i\geq 1}$. We will let $C^{\operatorname{(DEC)}}_{q,X}$ denote the best possible constant and call it the \emph{$q$-decoupling constant} of $X$. \blue{The decoupling property is closely related to the well-known UMD property. Recall that a Banach space $X$ is called a UMD space if for some (then all) $1<q<\infty$ there is a constant $C_{q,X}$ such that for any finite martingale difference sequence $(\xi_i)_{i\geq 1}$ in $L^q(\Omega;X)$ and any $\varepsilon_i\in \{-1,1\}$, $i\geq 1$,
\begin{equation}
\label{eqn:UMD}
\Big(\E\Big\|\sum_{i=1}^n \varepsilon_i \xi_i\Big\|_X^q\Big)^{1/q} \leq C_{q,X} \Big(\E\Big\|\sum_{i=1}^n \xi_i\Big\|_X^q\Big)^{1/q}.
\end{equation}
We denote the best constant in \eqref{eqn:UMD} by $C^{\operatorname{(UMD)}}_{q,X}$. If $X$ is a UMD Banach space, then it satisfies the $q$-decoupling property for any $1\leq q<\infty$. Moreover, in this case} one can also \emph{recouple}, meaning that for all
$1<q<\infty$ there is a constant $c_{q,X}$ such that for any
martingale difference sequence $(\xi_i)_{i\geq 1}$ and any
associated decoupled tangent sequence $(\theta_i)_{i\geq 1}$,
\begin{equation}
\label{eqn:recouple}
\Big(\E\Big\|\sum_{i=1}^n \theta_i\Big\|_X^q\Big)^{1/q} \leq c_{q,X} \Big(\E\Big\|\sum_{i=1}^n \xi_i\Big\|_X^q\Big)^{1/q}.
\end{equation}
Let us denote the best possible constant in this inequality by $C^{\operatorname{(REC)}}_{q,X}$ and call it the \emph{$q$-recoupling constant}. Conversely, if both (\ref{eqn:decouple}) and (\ref{eqn:recouple}) hold for some $1<q<\infty$, then $X$ must be a UMD space (in this case \eqref{eqn:decouple} and \eqref{eqn:recouple} hold for all choices of $q$). This equivalence is independently due to McConnell \cite{McC89} and Hitczenko \cite{HitUP}. In fact, it follows from their proof that $\max\{C^{\operatorname{(DEC)}}_{q,X},C^{\operatorname{(REC)}}_{q,X}\}\leq C^{\operatorname{(UMD)}}_{q,X}$. Finally, we will use that for any $1<r\leq q$, $C^{\operatorname{(DEC)}}_{q,X}\lesssim \frac{q}{r}C^{\operatorname{(DEC)}}_{r,X}$ (\cite[Theorem 4.1]{CoV11}) and $C^{\operatorname{(UMD)}}_{q,X}\lesssim \frac{q}{r}C^{\operatorname{(UMD)}}_{r,X}$ (\cite[Theorem 4.23]{HNV16}).\par  
With these tools we can now prove the main result of this section.
\begin{proof}[Proof of Theorem~\ref{thm:matrixBRintro}]
Let $C^{\operatorname{(DEC)}}_{q,p}$ and $C^{\operatorname{(REC)}}_{q,p}$ denote the $q$-decoupling and $q$-recoupling constants of the Banach space of $(\R^{p\times p},\|\cdot\|)$, which is finite dimensional and hence a UMD space. Moreover, we write $\E_{\cG} = \E(\cdot|\cG)$ for brevity. Let $({\THETA}_k)$ be the decoupled tangent sequence for the martingale difference sequence $(\XI_k)$. By the $q$-decoupling inequality,
\begin{align*}
\Big(\E\Big\|\sum_k \XI_k\Big\|^q\Big)^{1/q} & \leq C^{\operatorname{(DEC)}}_{q,p} \Big(\E\Big\|\sum_k {\THETA}_k\Big\|^q\Big)^{1/q}.
\end{align*}
Since the summands ${\THETA}_k$ are $\cG$-conditionally independent by \eqref{eq:G_independence} and $\cG$-mean zero by \eqref{eqn:tangent}, \eqref{eqn:CI}, and the fact that $(\XI_k)$ forms a martingale difference sequence, we can apply Theorem~\ref{thm:sumsRMs} conditionally to find, a.s.,
	\begin{align}
	\label{est:decoupledLqnorm}
	\Big(\E_{\cG}\Big\|\sum_{k=1}^n {\THETA}_k\Big\|^q \Big)^{1/q} & \leq C_{q,p}^{\operatorname{(MR)}} \max\Big\{\Big\|\Big(\sum_{k=1}^n \E_{\cG}(\THETA_k^T\THETA_k)\Big)^{1/2}\Big\|, \Big\|\Big(\sum_{k=1}^n \E_{\cG}(\THETA_k\THETA_k^T)\Big)^{1/2}\Big\|,\nonumber\\
	& \qquad \qquad\qquad\qquad\qquad\qquad \qquad\qquad \ \ \ \ \ 2C_{\frac{q}{2},p}^{\operatorname{(MR)}}\Big(\E_{\cG} \max_{1\leq k\leq n} \|{\THETA}_k\|^q\Big)^{1/q}\Big\}\nonumber\\
	& \leq C_{q,p}^{\operatorname{(MR)}} \max\Big\{\Big\|\Big(\sum_{k=1}^n \mathbb{E}_{k-1}(\XI_k^T\XI_k)\Big)^{1/2}\Big\|, \Big\|\Big(\sum_{k=1}^n \mathbb{E}_{k-1}(\XI_k\XI_k^T)\Big)^{1/2}\Big\|,\nonumber\\
	& \qquad \qquad\qquad\qquad\qquad\qquad \qquad \qquad \ \ \ \ \ 2C_{\frac{q}{2},p}^{\operatorname{(MR)}} \Big( \max_{1\leq k\leq n} \E_{k-1}\|{\XI}_k\|^q\Big)^{1/q}\Big\},
	\end{align}
where in the last step we used \eqref{eq:Lq_maxBound} (together with the assumption $q\geq \log(n)$) and moreover used that by the properties \eqref{eqn:CI} and \eqref{eqn:tangent} of a decoupled tangent sequence, 
$$\E_{\cG} (\THETA_k^T\THETA_k)  = \E_{k-1} (\XI_k^T\XI_k)$$
and 
$$\E_{\cG}\|{\THETA}_k\|^q = \E_{k-1}\|{\XI}_k\|^q.$$
Now take $L^q$-norms on both sides of \eqref{est:decoupledLqnorm}, apply the triangle inequality, and again use \eqref{eq:Lq_maxBound} to obtain the claimed estimate.\par
To estimate $C^{\operatorname{(DEC)}}_{q,p}$, recall that the operator norm on $\R^{p\times p}$ is equivalent to the Schatten $r$-norm for $r:=\log p$ up to absolute constants. Since $q\geq \log(p)\geq 1$, 
$$C^{\operatorname{(DEC)}}_{q,p}\simeq C^{\operatorname{(DEC)}}_{q,S^{\log(p)}} \lesssim \frac{q}{\log(p)}  C^{\operatorname{(DEC)}}_{\log(p),S^{\log(p)}} \leq \frac{q}{\log(p)} C^{\operatorname{(UMD)}}_{\log(p),S^{\log(p)}}$$
Since $C^{\operatorname{(UMD)}}_{r,S^r}\lesssim r$ for any $2\leq r<\infty$ (see \cite[Corollary 4.5]{Ran02}), 
we find $C^{\operatorname{(DEC)}}_{q,p}\lesssim q$ whenever $q\geq \log(p)$. This proves \eqref{eqn:matrixBRintroUpper}.\par
Proceeding analogously using the lower bound in Theorem~\ref{thm:sumsRMs}, we find 
		\begin{align*}
	C^{\operatorname{(REC)}}_{q,p} \Big(\E \Big\|\sum_{k=1}^n {\XI}_k\Big\|^q \Big)^{1/q} & \gtrsim \max\Big\{\Big(\E\Big\|\Big(\sum_{k=1}^n \mathbb{E}_{k-1}(\XI_k^T\XI_k)\Big)^{1/2}\Big\|^q\Big)^{1/q}, \Big(\E\Big\|\Big(\sum_{k=1}^n \mathbb{E}_{k-1}(\XI_k\XI_k^T)\Big)^{1/2}\Big\|^q\Big)^{1/q},\\
	& \qquad \qquad\qquad\qquad\qquad\qquad \qquad\ \ \ \ \ \max_{1\leq k\leq n} (\E\|{\XI}_k\|^q)^{1/q}\Big\}.
	\end{align*}
Since 
$$C^{\operatorname{(REC)}}_{q,p}\simeq C^{\operatorname{(REC)}}_{q,S^{\log(p)}}\leq  C^{\operatorname{(UMD)}}_{q,S^{\log(p)}} \lesssim \frac{q}{\log(p)} C^{\operatorname{(UMD)}}_{\log (p),S^{\log(p)}} \lesssim q,$$
our proof is complete.    
\end{proof}

%

\section*{Acknowledgements}

The authors were supported by the Deutsche Forschungsgemeinschaft (DFG, German Research Foundation) through the project CoCoMIMO funded within the priority program SPP 1798 Compressed Sensing in Information Processing (COSIP).

\appendix 
\section{Proof of Theorem \ref{thm:OperatorDitheredMaskDecoupled}}
\label{sec:ProofAppendix}
\blue{
\newJ{
To show Theorem \ref{thm:OperatorDitheredMaskDecoupled}, we \newS{first} need the following refined version of Lemma \ref{lem:linftyBiasEst}.
\begin{lemma}
	\label{lem:linftyBiasEstDecoupled}
	There exists constant\sjoerd{s $c_1,c_2>0$} depending only on $K$ such that the following holds. Let $\X$ be a mean-zero, $K$-subgaussian vector with covariance matrix $\johannes{\E\round{ \X\X^T }}{} = \SIGMA$. Let \sjoerd{$\LAMBDA \in \R_{>0}^{p\times p}$ be diagonal and let} $\Y = \LAMBDA \, \sign(\X + \LAMBDA\Tau)$ and $\bar{\Y}= \LAMBDA \, \sign(\X + \LAMBDA\bar{\Tau})$, where $\Tau,\bar{\Tau}$ are independent and uniformly distributed in $[-1,1]^p$ \sjoerd{and independent of $\X$}. Then,
	$$ \Big| \Big( \johannes{\E\round{\Y\bar{\Y}^T}}{} - \SIGMA \Big)_{ij} \Big| \sjoerd{\leq c_1} (\Lambda_{ii}\Lambda_{jj} + \Sigma_{ii}^{1/2}\Sigma_{jj}^{1/2}) e^{-c_2 \min\left\{ \frac{\Lambda_{ii}}{\Sigma_{ii}^{1/2}}, \frac{\Lambda_{jj}}{\Sigma_{jj}^{1/2}} \right\}^2}$$
	and 
	$$ \Big| \Big( \johannes{\E\round{\Y\Y^T}}{} - (\SIGMA-\diag(\SIGMA)+\LAMBDA^2) \Big)_{ij} \Big| \sjoerd{\leq c_1} (\Lambda_{ii}\Lambda_{jj} + \Sigma_{ii}^{1/2}\Sigma_{jj}^{1/2}) e^{-c_2 \min\left\{ \frac{\Lambda_{ii}}{\Sigma_{ii}^{1/2}}, \frac{\Lambda_{jj}}{\Sigma_{jj}^{1/2}} \right\}^2}$$
\end{lemma}
}
%
\newS{The proof relies on the following lemma. For $\lambda=\lambda'=1$ it is immediate from (the proof of) \cite[Lemma 16]{dirksen2021covariance}. The general statement follows by rescaling.}  
%
%
\newJ{
\begin{lemma} \label{lem:biasSignProd}
	Let $U,V$ be subgaussian random variables, let $\lambda,\lambda'>0$, and let $\sigma,\sigma'$ be independent, uniformly distributed in $[-1,1]$ and independent of $U$ and $V$. Then, 
	$$\left| \johannes{\E\round{\lambda\sign(U + \lambda\sigma) \cdot \lambda' \sign(V + \lambda'\sigma')} - \E\round{UV}}{} \right|
	\lesssim
	(\lambda \lambda' + \|U\|_{\psi_2} \|V\|_{\psi_2}) e^{-c \min\left\{ \frac{\lambda}{\| U \|_{\psi_2}}, \frac{\lambda'}{\| V \|_{\psi_2}} \right\}^2}.$$
\end{lemma}
}
\newJ{
\begin{proof}[Proof of Lemma \ref{lem:linftyBiasEstDecoupled}]
	Since $\X$ is $K$-subgaussian, for any $\ell\in[p]$,
	$$\|X_{\ell}\|_{\psi_2} = \|\langle \X,\mathbf{e}_{\ell}\rangle\|_{\psi_2} \leq K \johannes{(\mathbb{E}\langle \X,\mathbf{e}_{\ell}\rangle^2)^{1/2}}{} = K\SIGMA_{\ell\ell}^{1/2}.$$
	Lemma~\ref{lem:biasSignProd} applied for $U=X_i$ and $V=X_j$ yields 
	$$\abs{ \johannes{\E\round{ \Lambda_{ii} \, \sign(X_i + \Lambda_{ii} \tau_i) \Lambda_{jj} \, \sign(X_j + \Lambda_{jj} \bar\tau_j)}}{} - \Sigma_{ij} }{} \lesssim (\Lambda_{ii}\Lambda_{jj}+\sjoerdgreen{K^2} \Sigma_{ii}^{1/2}\Sigma_{jj}^{1/2}) e^{-c_2 \min\left\{ \frac{\Lambda_{ii}}{\Sigma_{ii}^{1/2}}, \frac{\Lambda_{jj}}{\Sigma_{jj}^{1/2}} \right\}^2}, $$ 
	for all $i,j\in [p]$ and  
	$$\abs{ \johannes{\E\round{ \Lambda_{ii} \, \sign(X_i + \Lambda_{ii} \tau_i) \Lambda_{jj} \,\sign(X_j + \Lambda_{jj}\tau_j)}}{} - \Sigma_{ij} }{} \lesssim (\Lambda_{ii}\Lambda_{jj}+\sjoerdgreen{K^2} \Sigma_{ii}^{1/2}\Sigma_{jj}^{1/2}) e^{-c_2 \min\left\{ \frac{\Lambda_{ii}}{\Sigma_{ii}^{1/2}}, \frac{\Lambda_{jj}}{\Sigma_{jj}^{1/2}} \right\}^2}$$
	whenever $i\neq j$. These two observations immediately imply the two statements. 
\end{proof}
}
\newJ{
Second, we have to control the various adaptive dithering scales, i.e., the entries of $\LAmbda_k$ defined in \eqref{eq:lambdaAdaptiveEntrywise}.
\begin{lemma}
	\label{lem:linfConcentrationDecoupled}
	There exists an absolute constants $c>0$ and such that the following holds. Let $\X \in \R^p$ be a $K$-subgaussian vector with $\E \X = \0$ and $\E(\X\X^\top) = \SIGMA \in \R^{p\times p}$. Let $\X_0,\dots,\X_n \overset{i.i.d.}{\sim} \X$ and let $\LAmbda_k$ be defined as in \eqref{eq:lambdaAdaptiveEntrywise}. Then, for $t > 0$ and any $i \in [p]$, $k \in [n]$,
	\begin{align}
		\label{eqn:lambdakSubgaussianEntrywise}
		\P{ | (\lambda_k)_i - \Sigma_{ii}^{1/2}| > t^{1/2} K^{1/2} \Sigma_{ii}^{1/2}}
		\le 2 e^{-ck\min\{t,t^2\}}.
	\end{align}
\end{lemma}
\begin{proof}
	Fix $i \in [p]$. Let $X_{\newS{ki}}$ and $\lambda_{\newS{ki}}$ \newS{denote} the $i$-th entry of $\X_k$ and $\LAmbda_k$, respectively. Since
	\begin{align*}
		\| X_{\newS{ki}}^2 \|_{\psi_1} = \| X_{\newS{ki}} \|_{\psi_2}^2 = \| \langle \X_{\newS{ki}}, \e_i \rangle \|_{\psi_2}^2 \le K^2  \E \langle \X_{\newS{k}},\e_i \rangle^2 = K^2 \Sigma_{ii}
	\end{align*}
\newS{and} $\E X_{\newS{ji}}^2 = \Sigma_{ii}$ \newS{for all $j\in[n]$},
   Bernstein's inequality (\newS{see}, e.g., \cite[Theorem 2.8.1]{vershynin2018high}) yields
	\begin{align*}
		\P{ |\lambda_{\newS{ki}}^2 - \Sigma_{ii}| > t K  \Sigma_{ii}}
		= \P{ \left| \sum_{j=0}^{k-1} (X_{\newS{ji}}^2 - \E X_{\newS{ji}}^2) \right| > k t K \Sigma_{ii} }
		\le 2 e^{-c \min\left\{ \frac{t^2 k^2 K^2 \Sigma_{ii}^2}{k K^2 \Sigma_{ii}^2}, \frac{t k K \Sigma_{ii}}{K \Sigma_{ii}} \right\}}
		= 2 e^{-ck \min\{ t^2,t \}},
	\end{align*}
     where $c > 0$ is an absolute constant.
     Since $\sqrt{a}-\sqrt{b} \le \sqrt{a-b}$, for $0 \le b \le a$, this yields the claim via
     \begin{align*}
     	\P{ | \lambda_{\newS{ki}} - \Sigma_{ii}^{1/2}| > t^{1/2} K^{1/2} \Sigma_{ii}^{1/2} }
     	\le \P{ | \lambda_{\newS{ki}}^2 - \Sigma_{ii}|^{1/2} >  t^{1/2} K^{1/2} \Sigma_{ii}^{1/2} }.
     \end{align*}
\end{proof}
}
\newJ{
\newS{Using} Lemmas \ref{lem:linftyBiasEstDecoupled} and \ref{lem:linfConcentrationDecoupled}, we can now derive the required bias control.
\begin{lemma}
	\label{lem:biasControlUnifiedDecoupled}
	There exists an absolute constant $c$ such that the following holds.
	Let $\X \in \mathbb R^p$ be a mean-zero, $K$-subgaussian vector with covariance matrix \sjoerd{$\SIGMA \in \R^{p\times p}$}.
	Let $\X_0,...,\X_n \overset{\mathrm{i.i.d.}}{\sim} \X$ and let $\LAMBDA_k$ be a diagonal matrix with $\diag(\LAMBDA_k) = \sqrt{\frac{4}{c_2} \log(k^2)} \LAmbda_k$, where $\LAmbda_k$ is defined in \eqref{eq:lambdaAdaptiveEntrywise} and $c_2$ is the \newS{constant} from Lemma \ref{lem:linftyBiasEstDecoupled}. 
	Then, for any $\theta\geq \max\{c, 4\log(n),4\log(p)\}$, we have with probability at least $1-4e^{-\theta}$
	$$\Big\| \sum_{k=1}^n \E_{k-1}( \tfrac{1}{n} \Y_k\bar{\Y}_k^T) - \SIGMA \Big\| \lesssim_{K} \theta^2 \log(p)\log(n) \frac{\tr(\SIGMA)}{n},$$
	where we abbreviate the scaled quantized samples appearing in \eqref{eq:AsymmetricEstimatorAdaptiveEntrywise} by $\Y_k = \LAMBDA_k \sign(\X_k + \LAMBDA_k \Tau_k)$ and $\bar{\Y}_k = \LAMBDA_k \sign(\X_k + \LAMBDA_k \bar{\Tau}_k)$.
\end{lemma}
As before, we extract the following technical observation from the proof of Lemma \ref{lem:biasControlUnifiedDecoupled} since we are going to re-use it later.
\begin{lemma}
	\label{lem:SumBoundsDecoupled}
	There is an absolute constant $c\ge1$ such that the following holds. Let $p > 2$ and let $\X \in \R^p$ be a $K$-subgaussian vector with $\E \X = \0$ and $\E(\X\X^\top) = \SIGMA \in \R^{p\times p}$. Let $\X_1,\dots,\X_n \overset{i.i.d.}{\sim} \X$ and let $\LAmbda_k$ be defined as in \eqref{eq:lambdaAdaptiveEntrywise}.
	Let $\LAMBDA_k$ be a diagonal matrix with $\diag(\LAMBDA_k) = \sqrt{\frac{4}{ c_2} \log(k^2)} \LAmbda_k$ for $k\geq 1$, where $c_2$ is the \newS{constant} from \newS{Lemma} \ref{lem:linftyBiasEstDecoupled}. 
 Then, for $\theta \ge c$ we have with probability at least $1-2e^{-\theta^2}$
	\begin{equation}
		\label{eqn:elemLambdaEstimateDecoupled}
		(\Lambda_k)_{ii} \lesssim \sqrt{ \theta K \Sigma_{ii} \log(p) \log(k))}, \qquad \text{for all $k \ge 1$ and $i \in [p]$,}
	\end{equation}
 \newS{and in particular
\begin{equation}
\label{eqn:lambdaTraceEst}
\tr(\LAMBDA_k^2)\lesssim \theta K\log(p) \log(k))\tr(\SIGMA) , \qquad \text{for all $k \ge 1$}.
\end{equation}
}	Moreover, for any $\theta\geq \max\{c, 4\log(n), 4\log(p)\}$, we have with probability at least $1-4e^{-\theta}$ that for any $(i,j) \in [p]^2$
	\begin{equation}
		\label{eqn:sumLambdaEstimate1Decoupled}
		\sum_{k=1}^n ((\Lambda_k)_{ii}(\Lambda_k)_{jj} + \Sigma_{ii}^{1/2}\Sigma_{jj}^{1/2}) e^{-c_2 \min\left\{ \frac{(\Lambda_k)_{ii}}{\Sigma_{ii}^{1/2}}, \frac{(\Lambda_k)_{jj}}{\Sigma_{jj}^{1/2}} \right\}^2}
		\lesssim_K \theta^2 \log(p)\log(n) \Sigma_{ii}^{1/2}\Sigma_{jj}^{1/2}.
	\end{equation}
    and 
    \begin{align}
    \label{eqn:sumLambdaEstimate2Decoupled}
    	 \left( \sum_{k=1}^n \frac{2 \tr(\LAMBDA_k^2) }{n^2}
    	\big( \tr(\LAMBDA_k^2) + \tr(\SIGMA) \big) e^{-c_2 \min\left\{ \frac{\Lambda_{ii}}{\Sigma_{ii}^{1/2}} \colon i \in [p] \right\}^2}  \right)^{1/2} 
    	\lesssim_K \frac{\theta^2}{n} \log(p) \log(n) \tr(\SIGMA).
    \end{align}
\end{lemma}
}
\begin{proof}
	\newJ{
	Assuming that $\theta$ is sufficiently large to guarantee $e^{-c \min\{\theta,\theta^2\}} \le \frac{1}{2}$, for $c>0$ being the constant from Lemma~\ref{lem:linfConcentrationDecoupled}, one can combine \eqref{eqn:lambdakSubgaussianEntrywise} in Lemma \ref{lem:linfConcentrationDecoupled} with a union bound over all $k\geq 1$ and all $i \in [p]$, to obtain with probability at least $1-2e^{-c\theta^2}$ 
	\begin{align} \label{eq:LambdaUBDecoupled}
		(\lambda_k)_i \lesssim \sqrt{ (\theta + \log(p)) K \Sigma_{ii}} \le \sqrt{ \theta K \Sigma_{ii} \log(p)},
	\end{align}
	for all $k \ge 1$ and all $i \in [p]$, where we used in the last step that $\theta, \log(p) \ge 1$ by assumption.
	This proves \eqref{eqn:elemLambdaEstimateDecoupled} and \eqref{eqn:lambdaTraceEst}.
}
\newJ{
	Define now $k_\theta = \min\{\lceil c' \theta K^2 \rceil,n\}$, where $c'>0$ only depends on the constant $c$ from Lemma \ref{lem:linfConcentrationDecoupled}. If $k_\theta < n$, we find by Lemma \ref{lem:linfConcentrationDecoupled} and a union bound over all $k_{\theta} < k \leq n$ and all $i \in [p]$ with probability at least 
	\begin{align*}
		2 \sum_{k = k_\theta + 1}^n p e^{ - \frac{ck}{16 K^2 } } 
		\le 2 n p e^{-\theta} 
		\le 2 e^{-\frac{1}{2} \theta}
	\end{align*}
	that 
	\begin{align}
		\label{eq:LambdaKLBDecoupled}
		(\lambda_k)_i \ge \frac{1}{2} \Sigma_{ii}^{1/2}, \qquad \text{for any $k_\theta < k \le n$ and $i \in [p]$}.
	\end{align}
	Note that we used in the union bound that $\theta - \log(n) - \log(p) \ge \frac{1}{2} \theta$ by assumption.
}
\newJ{
\newS{Let us now condition on these events. For any $(i,j) \in [p]^2$ we make the split}
	\begin{align*}
		&\sum_{k=1}^n ((\Lambda_k)_{ii}(\Lambda_k)_{jj} + \Sigma_{ii}^{1/2}\Sigma_{jj}^{1/2}) e^{-c_2 \min\left\{ \frac{(\Lambda_k)_{ii}}{\Sigma_{ii}^{1/2}}, \frac{(\Lambda_k)_{jj}}{\Sigma_{jj}^{1/2}} \right\}^2} \\
		& \qquad = \sum_{\substack{1 \le k \le k_\theta}} ((\Lambda_k)_{ii}(\Lambda_k)_{jj} + \Sigma_{ii}^{1/2}\Sigma_{jj}^{1/2}) e^{-c_2 \min\left\{ \frac{(\Lambda_k)_{ii}}{\Sigma_{ii}^{1/2}}, \frac{(\Lambda_k)_{jj}}{\Sigma_{jj}^{1/2}} \right\}^2} \\
		& \qquad\quad + \sum_{k=k_{\theta}+1}^n ((\Lambda_k)_{ii}(\Lambda_k)_{jj} + \Sigma_{ii}^{1/2}\Sigma_{jj}^{1/2}) e^{-c_2 \min\left\{ \frac{(\Lambda_k)_{ii}}{\Sigma_{ii}^{1/2}}, \frac{(\Lambda_k)_{jj}}{\Sigma_{jj}^{1/2}} \right\}^2},
	\end{align*}
	where we recall that $(\Lambda_k)_{ii} = \sqrt{\frac{4}{c_2} \log(k^2)} (\lambda_k)_i$.
	Observe that by \eqref{eqn:elemLambdaEstimateDecoupled}
	\begin{align*}
		\sum_{\substack{1 \le k \le k_\theta}} ((\Lambda_k)_{ii}(\Lambda_k)_{jj} + \Sigma_{ii}^{1/2}\Sigma_{jj}^{1/2}) e^{-c_2 \min\left\{ \frac{(\Lambda_k)_{ii}}{\Sigma_{ii}^{1/2}}, \frac{(\Lambda_k)_{jj}}{\Sigma_{jj}^{1/2}} \right\}^2}
		&\lesssim_K k_{\theta} \cdot \theta \Sigma_{ii}^{1/2}\Sigma_{jj}^{1/2} \log(p)\log(k_\theta) \\
		&\lesssim_K \theta^2 \log(p)\log(n) \Sigma_{ii}^{1/2}\Sigma_{jj}^{1/2}.
	\end{align*}
	Similarly, if $k_{\theta}<n$ then \eqref{eqn:elemLambdaEstimate} and \eqref{eq:LambdaKLBDecoupled} yield
	\begin{align*}
		\sum_{k=k_{\theta}+1}^n ((\Lambda_k)_{ii}(\Lambda_k)_{jj} + \Sigma_{ii}^{1/2}\Sigma_{jj}^{1/2}) e^{-c_2 \min\left\{ \frac{(\Lambda_k)_{ii}}{\Sigma_{ii}^{1/2}}, \frac{(\Lambda_k)_{jj}}{\Sigma_{jj}^{1/2}} \right\}^2}
		&\lesssim_K  \sum_{k=k_{\theta}+1}^n \theta \Sigma_{ii}^{1/2}\Sigma_{jj}^{1/2} \log(p)\log(k) \cdot e^{-\log(k^2)} \\
		&\lesssim \theta \Sigma_{ii}^{1/2}\Sigma_{jj}^{1/2} \log(p).
	\end{align*}
	Collecting these estimates we obtain \eqref{eqn:sumLambdaEstimate1Decoupled}. 
}
\newJ{
    For \eqref{eqn:sumLambdaEstimate2Decoupled} we use the same split of the sum at $k_\theta$. \newS{First, note that by \eqref{eqn:lambdaTraceEst}} 
    \begin{align*}
    	\left( \sum_{1\le k \le k_\theta} \frac{2 \tr(\LAMBDA_k^2) }{n^2}
    	\big( \tr(\LAMBDA_k^2) + \tr(\SIGMA) \big) e^{-c_2 \min\left\{ \frac{(\Lambda_k)_{ii}}{\Sigma_{ii}^{1/2}} \colon i \in [p] \right\}^2}  \right)^{1/2} 
    	&\lesssim_K \frac{1}{n} \left(  k_\theta \cdot \theta^2 \log(p)^2 \log(k_\theta)^2 \tr(\SIGMA)^2 \right)^{1/2} \\
    	&\lesssim_K \frac{\theta^2}{n} \log(p) \log(n) \tr(\SIGMA).
    \end{align*}
    Moreover, \newS{by} \eqref{eq:LambdaUBDecoupled} and \eqref{eq:LambdaKLBDecoupled}, we obtain that
    \begin{align*}
    	& \left( \sum_{k=k_\theta+1}^n \frac{2 \tr(\LAMBDA_k^2) }{n^2}
    	\big( \tr(\LAMBDA_k^2) + \tr(\SIGMA) \big) e^{-c_2 \min\left\{ \frac{(\Lambda_k)_{ii}}{\Sigma_{ii}^{1/2}} \colon i \in [p] \right\}^2}  \right)^{1/2}
    	\\
     &\qquad \lesssim_K \left( \sum_{k=k_\theta+1}^n \frac{\theta^2}{n^2} \log(p)^2 \log(k)^2 \tr(\SIGMA)^2
    	e^{-\log(k^2)}  \right)^{1/2} \\
    	&\qquad \le \frac{\theta}{n} \log(p) \tr(\SIGMA).
    \end{align*}
    Combining the last two bounds yields \eqref{eqn:sumLambdaEstimate2Decoupled}.
}
\end{proof}
\newJ{
\begin{proof}[Proof of Lemma \ref{lem:biasControlUnifiedDecoupled}]
	We compute that with probability at least $1-4e^{-\theta}$
	\begin{align}
		\label{eqn:normEstPosDecoupled}
		&\pnorm{ \sum_{k=1}^n \E_{k-1}( \tfrac{1}{n} \Y_k\bar{\Y}_k^T) - \SIGMA }{} \nonumber \\
		&=  \sup_{\v,\w \in \R^p \ : \ \|\v\|_2, \|\w\|_2 \leq 1} \abs{\sum_{i,j=1}^p \Big( \sum_{k=1}^n \E_{k-1}( \tfrac{1}{n} \Y_k\bar{\Y}_k^T) - \SIGMA \Big)_{i,j} v_i w_j }{} \nonumber\\
		&
		\le  \sup_{\v,\w \in \R^p \ : \ \|\v\|_2, \|\w\|_2 \leq 1} \sum_{i,j=1}^p \frac{1}{n} \sum_{k=1}^n \abs{ \Big( \E_{k-1}(  \Y_k\bar{\Y}_k^T) - \SIGMA \Big)_{i,j} }{} \abs{v_i w_j}{}  \nonumber\\
		&\lesssim_K \sup_{\v,\w \in \R^p \ : \ \|\v\|_2, \|\w\|_2 \leq 1} \frac{1}{n} \sum_{i,j=1}^p \sum_{k=1}^n ((\Lambda_k)_{ii}(\Lambda_k)_{jj} + \Sigma_{ii}^{1/2}\Sigma_{jj}^{1/2}) e^{-c_2 \min\left\{ \frac{(\Lambda_k)_{ii}}{\Sigma_{ii}^{1/2}}, \frac{(\Lambda_k)_{jj}}{\Sigma_{jj}^{1/2}} \right\}^2} \cdot \abs{v_i w_j}{}  \nonumber\\
		&\lesssim_K \sup_{\v,\w \in \R^p \ : \ \|\v\|_2, \|\w\|_2 \leq 1} \frac{1}{n} \sum_{i,j=1}^p  \theta^2 \log(p)\log(n) \Sigma_{ii}^{1/2}\Sigma_{jj}^{1/2} \cdot v_i w_j \nonumber\\
		&=  \frac{\theta^2 \log(p)\log(n)}{n} \cdot \| \diag(\SIGMA)^{\odot 1/2} \|_2^2 \nonumber\\
		&= \theta^2 \log(p)\log(n) \frac{\tr(\SIGMA)}{n},
	\end{align}
	where we applied Lemma~\ref{lem:linftyBiasEstDecoupled} conditionally and used Lemma~\ref{lem:SumBoundsDecoupled}.
\end{proof}
}
\newJ{
We \newS{now have} all the tools \newS{needed} to prove Theorem \ref{thm:OperatorDitheredMaskDecoupled}.
\begin{proof}[Proof of Theorem~\ref{thm:OperatorDitheredMaskDecoupled}]	
	The proof works analogously to the proof of Theorem \ref{thm:OperatorDitheredMask}.
	We again abbreviate $\Y_k = \LAMBDA_k \sign(\X_k + \LAMBDA_k \Tau_k)$ and $\bar{\Y}_k = \LAMBDA_k \sign(\X_k + \LAMBDA_k \bar{\Tau}_k)$, and set $\LAMBDA_k$ to be the diagonal matrix with $\diag(\LAMBDA_k) = \sqrt{\frac{4}{c_2} \log(k^2)} \LAmbda_k$ for $k\geq 1$, where $\LAmbda_k$ has been defined in \eqref{eq:lambdaAdaptiveEntrywise} and $c_2$ is the constant from Lemmas~\ref{lem:linftyBiasEstDecoupled}. 
	We also use \newS{a split similar to} \eqref{eqn:splitExpDithered}, but set $\M = \boldsymbol{1}$ and replace the global dithering scale $\tilde\lambda_k$ by the entry-wise scales $\LAMBDA_k$, i.e., we estimate 
	\begin{align}
		\label{eqn:splitExpDitheredDecoupled}
	    \pnorm{\newS{\hat{\SIGMA}'_n} -  \SIGMA}{}
		\le \pnorm{ \newS{\hat{\SIGMA}'_n} - \sum_{k=1}^n \E_{k-1}( \tfrac{1}{n} \Y_k\bar{\Y}_k^T) }{} + \pnorm{ \sum_{k=1}^n \E_{k-1}( \tfrac{1}{n} \Y_k\bar{\Y}_k^T) - \SIGMA }{},
	\end{align}
	Applying Lemma~\ref{lem:biasControlUnifiedDecoupled} to the second term on the right hand side of  \eqref{eqn:splitExpDitheredDecoupled} yields
	\begin{align}
		\label{eq:SecondTermBoundLoperatorDecoupled}
		\left\| \sum_{k=1}^n \E_{k-1}( \tfrac{1}{n} \Y_k\bar{\Y}_k^T) - \SIGMA \right\|  \lesssim_{K} \theta^2 \log(p)\log(n) \frac{\tr(\SIGMA)}{n}.
	\end{align} 
	To complete the proof, we will estimate the first term on the right hand side of the modified version of \eqref{eqn:splitExpDitheredDecoupled} using Theorem~\ref{thm:matrixBRintro}. To this end, define
	\begin{align*}
		{\XI}_k 
		= \frac{1}{n} \round{ \Y_k\bar{\Y}_k^T - \E_{k-1}\round{\Y_k\bar{\Y}_k^T} }{}
		\quad \text{ \sjoerd{so} that } \quad
		\newS{\hat{\SIGMA}'_n} - \sum_{k=1}^n \E_{k-1}( \tfrac{1}{n} \Y_k\bar{\Y}_k^T) = \sum_{k=1}^n \XI_k.
	\end{align*}
	Since $\E\|\XI_k\| < \infty$ and $\E_{k-1} \XI_k = \0$ by construction, the sequence $\XI_1,\dots,\XI_n$ forms a matrix martingale difference sequence with respect to the filtration \eqref{eqn:filtrationDef} and we can hence apply Theorem~\ref{thm:matrixBRintro}. Let us now estimate the quantities appearing in the bounds of Theorem \ref{thm:matrixBRintro}. For any $1\leq k\leq n$,
	\begin{align*}
		\|{\XI}_k\| = \frac{1}{n} \| \Y_k\bar{\Y}_k^T - \E_{k-1}\round{\Y_k\bar{\Y}_k^T} \| 
		\le \frac{2}{n} \tr(\LAMBDA_k^2),
	\end{align*}
	where we used that 
	\begin{align*}
		\| \Y_k\bar{\Y}_k^T \| 
		& = \|\newS{\LAMBDA_k \sign(\X_k + \LAMBDA_k \Tau_k)) \sign(\X_k + \LAMBDA_k \bar\Tau_k)^T \LAMBDA_k^T }\| \newS{\leq \tr(\LAMBDA_k^2).}
	\end{align*}
    \newS{Combining \eqref{eqn:lambdaTraceEst} and Lemma~\ref{lem:LqtoTailBound} yields}
	\begin{equation}
		\label{eqn:XikMaxNormEstimateLqDecoupled}
		\max_{1\leq k\leq n} (\E\|{\XI}_k\|^q)^{1/q} \lesssim \frac{1}{n} \max_{1\leq k\leq n} \| \tr(\LAMBDA_k^2) \|_{L^q}\lesssim_K  \log(p) \log(n) \tr(\SIGMA) \frac{\sqrt{q}}{n}.
	\end{equation} 
	Next, using that $(\Y_k,\bar{\Y}_k)$ and $(\bar{\Y}_k,\Y_k)$ are identically distributed, we get 
	\begin{align*}
		& \Big\|\Big(\sum_{k=1}^n \mathbb{E}_{k-1}(\XI_k^T\XI_k) \Big)^{1/2}\Big\| \\
		&\qquad = \pnorm{ \sum_{k=1}^n \frac{1}{n^2} \left( \mathbb{E}_{k-1}[(\Y_k\bar{\Y}_k^T)^T(\Y_k\bar{\Y}_k^T)] -  \Big( \johannes{\E_{k-1}\round{ \Y_k\bar{\Y}_k^T }}{} \Big)^T  \johannes{\E_{k-1}\round{ \Y_k\bar{\Y}_k^T }}{} \right) }{}^{1/2}\\
		&\qquad = \pnorm{ \sum_{k=1}^n \frac{1}{n^2} \left( \mathbb{E}_{k-1}[(\bar{\Y}_k\Y_k^T)^T(\bar{\Y}_k\Y_k^T)] -  \Big( \johannes{\E_{k-1}\round{ \bar{\Y}_k\Y_k^T }}{} \Big)^T  \johannes{\E_{k-1}\round{ \bar{\Y}_k\Y_k^T }}{} \right) }{}^{1/2}
	\end{align*}
	Interchanging the roles of $\Y_k$ and $\bar{\Y}_k$ yields
	\begin{align*}
		& \Big\|\Big(\sum_{k=1}^n \mathbb{E}_{k-1}(\XI_k\XI_k^T) \Big)^{1/2}\Big\| \\
		& \qquad = \pnorm{ \sum_{k=1}^n \frac{1}{n^2} \left( \mathbb{E}_{k-1}[(\bar{\Y}_k\Y_k^T)^T(\bar{\Y}_k\Y_k^T)] -  \Big( \johannes{\E_{k-1}\round{ \bar{\Y}_k\Y_k^T }}{} \Big)^T  \johannes{\E_{k-1}\round{ \bar{\Y}_k\Y_k^T }}{} \right) }{}^{1/2}.
	\end{align*}
	Since by \eqref{eq:Kadison}
	$$\Big( \johannes{\E_{k-1}\round{ \bar{\Y}_k\Y_k^T }}{} \Big)^T \johannes{\E_{k-1}\round{ \bar{\Y}_k\Y_k^T }}{} 
	\preceq \E_{k-1} \Big( ( \bar{\Y}_k\Y_k^T )^T (\bar{\Y}_k\Y_k^T ) \Big) $$ 
	we find
	\begin{align*}
		& \max\left\{ \Big\|\Big(\sum_{k=1}^n \mathbb{E}_{k-1}(\XI_k^T\XI_k) \Big)^{1/2}\Big\|, \Big\|\Big(\sum_{k=1}^n \mathbb{E}_{k-1}(\XI_k\XI_k^T) \Big)^{1/2}\Big\| \right\} \\
		&\quad \leq \left( \sum_{k=1}^n \frac{2 }{n^2} \pnorm{ \E_{k-1} \Big( ( \bar{\Y}_k\Y_k^T )^T (\bar{\Y}_k\Y_k^T ) \Big)}{} \right)^{1/2}
	    = \left( \sum_{k=1}^n \frac{2 \tr(\LAMBDA_k^2) }{n^2} \pnorm{ \E_{k-1} \Big( \Y_k \Y_k^T \Big)}{} \right)^{1/2}\\
		& \quad \leq \left( \sum_{k=1}^n \frac{2 \tr(\LAMBDA_k^2) }{n^2} \pnorm{  \johannes{\E_{k-1}\round{\Y_k\Y_k^T}}{} - (\SIGMA-\diag(\SIGMA)+ \LAMBDA_k^2) }{} \right. \\
		& \quad \qquad \qquad \qquad \qquad \qquad \qquad \qquad \qquad + \left. \sum_{k=1}^n \frac{2 \tr(\LAMBDA_k^2) }{n^2} \pnorm{  \SIGMA-\diag(\SIGMA)+\LAMBDA_k^2 }{} \right)^{1/2}.
	\end{align*}
	By the same reasoning as in \eqref{eqn:normEstPosDecoupled} together with Lemma~\ref{lem:linftyBiasEstDecoupled} and Lemma \ref{lem:SumBoundsDecoupled}, we obtain with probability at least $1-4e^{-\theta}$ that
	\begin{align*}
		&\left( \sum_{k=1}^n \frac{2 \tr(\LAMBDA_k^2) }{n^2} \pnorm{  \johannes{\E_{k-1}\round{\Y_k\Y_k^T}}{} - (\SIGMA-\diag(\SIGMA)+ \LAMBDA_k^2) }{} \right)^{1/2} \\
		&\le  \left( \sum_{k=1}^n \frac{2 \tr(\LAMBDA_k^2) }{n^2} \Big( \sup_{\v^k,\w^k \in \R^p \ : \ \|\v^k\|_2, \|\w^k\|_2 \leq 1} \sum_{i,j=1}^p
		(\Lambda_{ii}\Lambda_{jj} + \Sigma_{ii}^{1/2}\Sigma_{jj}^{1/2}) e^{-c_2 \min\left\{ \frac{\Lambda_{ii}}{\Sigma_{ii}^{1/2}}, \frac{\Lambda_{jj}}{\Sigma_{jj}^{1/2}} \right\}^2} \cdot |u_i^k v_j^k| \Big) \right)^{1/2} \\
		&\le  \left( \sum_{k=1}^n \frac{2 \tr(\LAMBDA_k^2) }{n^2}
		\big( \tr(\LAMBDA_k^2) + \tr(\SIGMA) \big) e^{-c_2 \min\left\{ \frac{\Lambda_{ii}}{\Sigma_{ii}^{1/2}} \colon i \in [p] \right\}^2}  \right)^{1/2} \\
		&\lesssim_K \frac{\theta^2}{n} \log(p) \log(n) \tr(\SIGMA).
	\end{align*}
	Moreover, on the same event we find by Lemma \ref{lem:SumBoundsDecoupled} that
	\begin{align*}
		\left( \sum_{k=1}^n \frac{2 \tr(\LAMBDA_k^2) }{n^2} \pnorm{  \SIGMA-\diag(\SIGMA)+\LAMBDA_k^2 }{} \right)^{1/2} 
		&\lesssim \left( \sum_{k=1}^n \frac{2 \tr(\LAMBDA_k^2) }{n^2} (\pnorm{  \SIGMA}{} + \pnorm{\LAMBDA_k^2}{} ) \right)^{1/2} \\
		&\lesssim_K \frac{\theta}{\sqrt{n}} \log(p) \log(n) \tr(\SIGMA)^{1/2} \| \SIGMA \|^{1/2}.
 	\end{align*}
	Thus, we obtain with probability at least $1-4e^{-\theta}$ that for any $\theta \ge \max\{c, 2\log(p),2\log(n)\}$ 
	\begin{align}
		\label{eqn:XikSquareFunEstimateDecoupled}
		& \max\left\{ \Big\|\Big(\sum_{k=1}^n \mathbb{E}_{k-1}(\XI_k^T\XI_k) \Big)^{1/2}\Big\|, \Big\|\Big(\sum_{k=1}^n \mathbb{E}_{k-1}(\XI_k\XI_k^T) \Big)^{1/2}\Big\| \right\} \nonumber\\
		& \quad \lesssim_K \frac{\theta}{\sqrt{n}} \log(p) \log(n) \tr(\SIGMA)^{1/2} \| \SIGMA \|^{1/2} + \frac{\theta^2}{n} \log(p) \log(n) \tr(\SIGMA).
	\end{align}
	By translating these high-probability estimates via Lemma \ref{lem:LqtoTailBound} into moment bounds, we find for $q \ge \max\{\log(p),\log(n)\}$ that
	\begin{align}
		\label{eqn:XikSquareFunEstimateLqDecoupled}
		& \max\left\{\left(\mathbb{E}\left\|\left(\sum_{k=1}^n \mathbb{E}_{k-1}(\XI_k^T\XI_k) \right)^{1/2}\right\|^q\right)^{1/q}, \left(\mathbb{E}\left\|\left(\sum_{k=1}^n \mathbb{E}_{k-1}(\XI_k\XI_k^T) \right)^{1/2}\right\|^q\right)^{1/q} \right\} \nonumber \\
		& \quad \lesssim_K \| \SIGMA \| \log(p) \log(n) \left( \sqrt{\frac{q^2}{n} \frac{\tr(\SIGMA)}{\| \SIGMA \|}} + \frac{q^2}{n} \frac{\tr(\SIGMA)}{\| \SIGMA \|} \right).
	\end{align}
	By using Theorem~\ref{thm:matrixBRintro} and inserting \eqref{eqn:XikMaxNormEstimateLqDecoupled} and \eqref{eqn:XikSquareFunEstimateLqDecoupled} we finally get for any $q \ge \max\{\log(p),\log(n)\}$ that 
	\begin{align*}
		& \left(\mathbb{E}\left\| \newS{\hat{\SIGMA}'_n} - \sum_{k=1}^n \E_{k-1}( \tfrac{1}{n} \Y_k\bar{\Y}_k^T) \right\|^q\right)^{1/q} \\
		& \qquad \lesssim_K q^{3/2} \| \SIGMA \| \log(p) \log(n) \left( \sqrt{\frac{q^2}{n} \frac{\tr(\SIGMA)}{\| \SIGMA \|}} + \frac{q^2}{n} \frac{\tr(\SIGMA)}{\| \SIGMA \|} \right).
	\end{align*}
	The result now immediately follows by applying Lemma~\ref{lem:LqtoTailBound} to the previous moment bound and inserting the resulting tail estimate with \eqref{eq:SecondTermBoundLoperatorDecoupled} into the right-hand side of \eqref{eqn:splitExpDitheredDecoupled}. 
\end{proof}
}
}

\bibliography{mybib}{}

\begin{thebibliography}{10}

\bibitem{ALP14}
Albert Ai, Alex Lapanowski, Yaniv Plan, and Roman Vershynin.
\newblock One-bit compressed sensing with non-{G}aussian measurements.
\newblock {\em Linear Algebra Appl.}, 441:222--239, 2014.

\bibitem{bar2002doa}
Ofer Bar-Shalom and Anthony~J Weiss.
\newblock {DOA} estimation using one-bit quantized measurements.
\newblock {\em IEEE Transactions on Aerospace and Electronic Systems},
  38(3):868--884, 2002.

\bibitem{BFN17}
Richard~G Baraniuk, Simon Foucart, Deanna Needell, Yaniv Plan, and Mary
  Wootters.
\newblock Exponential decay of reconstruction error from binary measurements of
  sparse signals.
\newblock {\em IEEE Transactions on Information Theory}, 63(6):3368--3385,
  2017.

\bibitem{bickel2008covariance}
Peter~J Bickel and Elizaveta Levina.
\newblock Covariance regularization by thresholding.
\newblock {\em The Annals of Statistics}, 36(6):2577--2604, 2008.

\bibitem{BoB08}
Petros~T Boufounos and Richard~G Baraniuk.
\newblock 1-bit compressive sensing.
\newblock In {\em 2008 42nd Annual Conference on Information Sciences and
  Systems}, pages 16--21. IEEE, 2008.

\bibitem{BJK15}
Petros~T Boufounos, Laurent Jacques, Felix Krahmer, and Rayan Saab.
\newblock Quantization and compressive sensing.
\newblock In {\em Compressed Sensing and its Applications: MATHEON Workshop
  2013}, pages 193--237. Springer, 2015.

\bibitem{Bur73}
Donald Burkholder.
\newblock Distribution function inequalities for martingales.
\newblock {\em The Annals of Probability}, 1:19--42, 1973.

\bibitem{cai2013optimal}
T~Tony Cai, Zhao Ren, and Harrison~H Zhou.
\newblock Optimal rates of convergence for estimating {T}oeplitz covariance
  matrices.
\newblock {\em Probability Theory and Related Fields}, 156(1-2):101--143, 2013.

\bibitem{cai2010optimal}
T~Tony Cai, Cun-Hui Zhang, and Harrison~H Zhou.
\newblock Optimal rates of convergence for covariance matrix estimation.
\newblock {\em The Annals of Statistics}, 38(4):2118--2144, 2010.

\bibitem{CaZ13}
Tony Cai and Wen-Xin Zhou.
\newblock A max-norm constrained minimization approach to 1-bit matrix
  completion.
\newblock {\em J. Mach. Learn. Res.}, 14(1):3619--3647, 2013.

\bibitem{chapeau2008fisher}
Fran{\c{c}}ois Chapeau-Blondeau, Solenna Blanchard, and David Rousseau.
\newblock Fisher information and noise-aided power estimation from one-bit
  quantizers.
\newblock {\em Digital Signal Processing}, 18(3):434--443, 2008.

\bibitem{chen2023parameter}
Junren Chen and Michael~K Ng.
\newblock A parameter-free two-bit covariance estimator with improved operator
  norm error rate.
\newblock {\em arXiv preprint arXiv:2308.16059}, 2023.

\bibitem{chen2022high}
Junren Chen, Cheng-Long Wang, Michael~K Ng, and Di~Wang.
\newblock High dimensional statistical estimation under uniformly dithered
  one-bit quantization.
\newblock {\em IEEE Transactions on Information Theory}, 2023.

\bibitem{chen2012masked}
Richard~Y Chen, Alex Gittens, and Joel~A Tropp.
\newblock The masked sample covariance estimator: an analysis using matrix
  concentration inequalities.
\newblock {\em Information and Inference: A Journal of the IMA}, 1(1):2--20,
  2012.

\bibitem{choi2016near}
Junil Choi, Jianhua Mo, and Robert~W Heath.
\newblock Near maximum-likelihood detector and channel estimator for uplink
  multiuser massive {MIMO} systems with one-bit {ADCs}.
\newblock {\em IEEE Transactions on Communications}, 64(5):2005--2018, 2016.

\bibitem{CoV11}
Sonja Cox and Mark Veraar.
\newblock Vector-valued decoupling and the {B}urkholder--{D}avis--{G}undy
  inequality.
\newblock {\em Illinois Journal of Mathematics}, 55(1):343--375, 2011.

\bibitem{dahmen2000structured}
J{\"o}rg Dahmen, Daniel Keysers, Michael Pitz, and Hermann Ney.
\newblock Structured covariance matrices for statistical image object
  recognition.
\newblock In {\em Mustererkennung 2000}, pages 99--106. Springer, 2000.

\bibitem{DPB14}
Mark~A Davenport, Yaniv Plan, Ewout Van Den~Berg, and Mary Wootters.
\newblock 1-bit matrix completion.
\newblock {\em Information and Inference: A Journal of the IMA}, 3(3):189--223,
  2014.

\bibitem{deng2020model}
Lei Deng, Guoqi Li, Song Han, Luping Shi, and Yuan Xie.
\newblock Model compression and hardware acceleration for neural networks: A
  comprehensive survey.
\newblock {\em Proceedings of the IEEE}, 108(4):485--532, 2020.

\bibitem{Dir14}
Sjoerd Dirksen.
\newblock It{\^o} isomorphisms for {$L^p$}-valued {P}oisson stochastic
  integrals.
\newblock {\em The Annals of Probability}, 42(6):2595--2643, 2014.

\bibitem{Dir19}
Sjoerd Dirksen.
\newblock Quantized compressed sensing: A survey.
\newblock In {\em Compressed Sensing and Its Applications: Third International
  MATHEON Conference 2017}, pages 67--95. Applied and Numerical Harmonic
  Analysis. Birkh{\"a}user, Cham, 2019.

\bibitem{DJR20}
Sjoerd Dirksen, Hans~Christian Jung, and Holger Rauhut.
\newblock One-bit compressed sensing with partial gaussian circulant matrices.
\newblock {\em Information and Inference: A Journal of the IMA}, 9(3):601--626,
  2020.

\bibitem{dirksen2021covariance}
Sjoerd Dirksen, Johannes Maly, and Holger Rauhut.
\newblock Covariance estimation under one-bit quantization.
\newblock {\em The Annals of Statistics}, 50(6):3538--3562, 2022.

\bibitem{DiM18a}
Sjoerd Dirksen and Shahar Mendelson.
\newblock Non-{G}aussian hyperplane tessellations and robust one-bit compressed
  sensing.
\newblock {\em Journal of the European Mathematical Society}, 23(9):2913--2947,
  2021.

\bibitem{DiM23}
Sjoerd Dirksen and Shahar Mendelson.
\newblock Robust one-bit compressed sensing with partial circulant matrices.
\newblock {\em The Annals of Applied Probability}, 33(3):1874--1903, 2023.

\bibitem{DiY19}
Sjoerd Dirksen and Ivan Yaroslavtsev.
\newblock {$L^q$}-valued {B}urkholder--{R}osenthal inequalities and sharp
  estimates for stochastic integrals.
\newblock {\em Proceedings of the London Mathematical Society},
  119(6):1633--1693, 2019.

\bibitem{EMY22}
Arian Eamaz, Kumar~Vijay Mishra, Farhang Yeganegi, and Mojtaba Soltanalian.
\newblock Uno: Unlimited sampling meets one-bit quantization.
\newblock {\em ArXiv:2301.10155}, 2022.

\bibitem{EMY23}
Arian Eamaz, Kumar~Vijay Mishra, Farhang Yeganegi, and Mojtaba Soltanalian.
\newblock Unlimited sampling via one-bit quantization.
\newblock In {\em 2023 International Conference on Sampling Theory and
  Applications (SampTA)}, pages 1--5. IEEE, 2023.

\bibitem{EYN23b}
Arian Eamaz, Farhang Yeganegi, Deanna Needell, and Mojtaba Soltanalian.
\newblock Harnessing the power of sample abundance: Theoretical guarantees and
  algorithms for accelerated one-bit sensing.
\newblock {\em ArXiv:2308.00695}, 2023.

\bibitem{EYN23}
Arian Eamaz, Farhang Yeganegi, Deanna Needell, and Mojtaba Soltanalian.
\newblock One-bit quadratic compressed sensing: From sample abundance to linear
  feasibility.
\newblock {\em ArXiv:2303.09594}, 2023.

\bibitem{eamaz2021modified}
Arian Eamaz, Farhang Yeganegi, and Mojtaba Soltanalian.
\newblock Modified arcsine law for one-bit sampled stationary signals with
  time-varying thresholds.
\newblock In {\em ICASSP 2021-2021 IEEE International Conference on Acoustics,
  Speech and Signal Processing (ICASSP)}, pages 5459--5463. IEEE, 2021.

\bibitem{eamaz2022covariance}
Arian Eamaz, Farhang Yeganegi, and Mojtaba Soltanalian.
\newblock Covariance recovery for one-bit sampled non-stationary signals with
  time-varying sampling thresholds.
\newblock {\em IEEE Transactions on Signal Processing}, 70:5222--5236, 2022.

\bibitem{EYS22}
Arian Eamaz, Farhang Yeganegi, and Mojtaba Soltanalian.
\newblock One-bit phase retrieval: More samples means less complexity?
\newblock {\em IEEE Transactions on Signal Processing}, 70:4618--4632, 2022.

\bibitem{eamaz2023covariance}
Arian Eamaz, Farhang Yeganegi, and Mojtaba Soltanalian.
\newblock Covariance recovery for one-bit sampled stationary signals with
  time-varying sampling thresholds.
\newblock {\em Signal Processing}, 206:108899, 2023.

\bibitem{EYS23}
Arian Eamaz, Farhang Yeganegi, and Mojtaba Soltanalian.
\newblock Matrix completion from one-bit dither samples.
\newblock {\em ArXiv:2310.03224}, 2023.

\bibitem{fang2010adaptive}
Jun Fang and Hongbin Li.
\newblock Adaptive distributed estimation of signal power from one-bit
  quantized data.
\newblock {\em IEEE Transactions on Aerospace and Electronic Systems},
  46(4):1893--1905, 2010.

\bibitem{Fou17}
Simon Foucart.
\newblock {\em Flavors of Compressive Sensing}, pages 61--104.
\newblock Springer International Publishing, Cham, 2017.

\bibitem{foucart2013compressed}
Simon Foucart and Holger Rauhut.
\newblock {\em {A Mathematical Introduction to Compressive Sensing}}.
\newblock Birkhäuser, New York, NY, 2013.

\bibitem{gersho2012vector}
Allen Gersho and Robert~M Gray.
\newblock {\em Vector quantization and signal compression}, volume 159.
\newblock Springer Science \& Business Media, 2012.

\bibitem{gholami2021survey}
Amir Gholami, Sehoon Kim, Zhen Dong, Zhewei Yao, Michael~W Mahoney, and Kurt
  Keutzer.
\newblock A survey of quantization methods for efficient neural network
  inference.
\newblock {\em arXiv:2103.13630}, 2021.

\bibitem{gray1987oversampled}
R~Gray.
\newblock {Oversampled Sigma-Delta modulation}.
\newblock {\em IEEE Transactions on Communications}, 35(5):481--489, 1987.

\bibitem{gray1998quantization}
Robert~M. Gray and David~L. Neuhoff.
\newblock Quantization.
\newblock {\em IEEE Transactions on Information Theory}, 44(6):2325--2383,
  1998.

\bibitem{guo2018survey}
Yunhui Guo.
\newblock A survey on methods and theories of quantized neural networks.
\newblock {\em arXiv:1808.04752}, 2018.

\bibitem{haghighatshoar2018low}
Saeid Haghighatshoar and Giuseppe Caire.
\newblock Low-complexity massive {MIMO} subspace estimation and tracking from
  low-dimensional projections.
\newblock {\em IEEE Transactions on Signal Processing}, 66(7):1832--1844, 2018.

\bibitem{HitUP}
P.~Hitczenko.
\newblock On tangent sequences of {UMD}-space valued random vectors.
  {U}npublished manuscript.

\bibitem{Hit94}
P.~Hitczenko.
\newblock On a domination of sums of random variables by sums of conditionally
  independent ones.
\newblock {\em Ann. Probab.}, 22(1):453--468, 1994.

\bibitem{HNV16}
Tuomas Hyt{\"o}nen, Jan Van~Neerven, Mark Veraar, and Lutz Weis.
\newblock {\em Analysis in Banach spaces}, volume~12.
\newblock Springer, 2016.

\bibitem{jacovitti1994estimation}
Giovanni Jacovitti and Alessandro Neri.
\newblock Estimation of the autocorrelation function of complex {G}aussian
  stationary processes by amplitude clipped signals.
\newblock {\em IEEE Transactions on Information Theory}, 40(1):239--245, 1994.

\bibitem{JLB13}
L.~{Jacques}, J.~N. {Laska}, P.~T. {Boufounos}, and R.~G. {Baraniuk}.
\newblock Robust 1-bit compressive sensing via binary stable embeddings of
  sparse vectors.
\newblock {\em IEEE Transactions on Information Theory}, 59(4):2082--2102,
  2013.

\bibitem{johnson1990matrix}
Charles~R Johnson.
\newblock {\em Matrix theory and applications}, volume~40.
\newblock American Mathematical Soc., 1990.

\bibitem{JMP19}
Hans~Christian Jung, Johannes Maly, Lars Palzer, and Alexander Stollenwerk.
\newblock Quantized compressed sensing by rectified linear units.
\newblock {\em IEEE Transactions on Information Theory}, 67(6):4125--4149,
  2021.

\bibitem{JuX03}
Marius Junge and Quanhua Xu.
\newblock Noncommutative burkholder/rosenthal inequalities.
\newblock {\em The Annals of Probability}, 31(2):948--995, 2003.

\bibitem{kabanava2017masked}
Maryia Kabanava and Holger Rauhut.
\newblock Masked {T}oeplitz covariance estimation.
\newblock {\em arXiv:1709.09377}, 2017.

\bibitem{ke2019user}
Yuan Ke, Stanislav Minsker, Zhao Ren, Qiang Sun, and Wen-Xin Zhou.
\newblock User-friendly covariance estimation for heavy-tailed distributions.
\newblock {\em Statistical Science}, 34(3):454--471, 2019.

\bibitem{knudson2016one}
Karin Knudson, Rayan Saab, and Rachel Ward.
\newblock One-bit compressive sensing with norm estimation.
\newblock {\em IEEE Transactions on Information Theory}, 62(5):2748--2758,
  2016.

\bibitem{koltchinskii2017}
Vladimir Koltchinskii and Karim Lounici.
\newblock Concentration inequalities and moment bounds for sample covariance
  operators.
\newblock {\em Bernoulli}, 23:110--133, 2017.

\bibitem{krim1996two}
Hamid Krim and Mats Viberg.
\newblock Two decades of array signal processing research: the parametric
  approach.
\newblock {\em IEEE signal processing magazine}, 13(4):67--94, 1996.

\bibitem{KwW92}
S.~Kwapie{\'n} and W.~Woyczy{\'n}ski.
\newblock {\em Random series and stochastic integrals: single and multiple}.
\newblock Probability and its Applications. Birkh\"auser Boston Inc., Boston,
  MA, 1992.

\bibitem{ledoit2003improved}
Olivier Ledoit and Michael Wolf.
\newblock Improved estimation of the covariance matrix of stock returns with an
  application to portfolio selection.
\newblock {\em Journal of empirical finance}, 10(5):603--621, 2003.

\bibitem{levina2012partial}
Elizaveta Levina and Roman Vershynin.
\newblock Partial estimation of covariance matrices.
\newblock {\em Probability theory and related fields}, 153(3-4):405--419, 2012.

\bibitem{li2017channel}
Yongzhi Li, Cheng Tao, Gonzalo Seco-Granados, Amine Mezghani, A~Lee
  Swindlehurst, and Liu Liu.
\newblock Channel estimation and performance analysis of one-bit massive {MIMO}
  systems.
\newblock {\em IEEE Transactions on Signal Processing}, 65(15):4075--4089,
  2017.

\bibitem{liu2021one}
Chun-Lin Liu and Zi-Min Lin.
\newblock One-bit autocorrelation estimation with non-zero thresholds.
\newblock In {\em ICASSP 2021-2021 IEEE International Conference on Acoustics,
  Speech and Signal Processing (ICASSP)}, pages 4520--4524. IEEE, 2021.

\bibitem{lybrand2021greedy}
Eric Lybrand and Rayan Saab.
\newblock A greedy algorithm for quantizing neural networks.
\newblock {\em Journal of Machine Learning Research}, 22(156):1--38, 2021.

\bibitem{maly2022simple}
Johannes Maly and Rayan Saab.
\newblock A simple approach for quantizing neural networks.
\newblock {\em Applied and Computational Harmonic Analysis}, 66:138--150, 2023.

\bibitem{Maly2022}
Johannes Maly, Tianyu Yang, Sjoerd Dirksen, Holger Rauhut, and Giuseppe Caire.
\newblock {\em New Challenges in Covariance Estimation: Multiple Structures and
  Coarse Quantization}, pages 77--104.
\newblock Springer International Publishing, Cham, 2022.

\bibitem{McC89}
T.~R. McConnell.
\newblock Decoupling and stochastic integration in {UMD} {B}anach spaces.
\newblock {\em Probab. Math. Statist.}, 10(2):283--295, 1989.

\bibitem{mendelson2018robust}
Shahar Mendelson and Nikita Zhivotovskiy.
\newblock Robust covariance estimation under {L}$_4$-{L}$_2$ norm equivalence.
\newblock {\em Annals of Statistics}, 48(3):1648--1664, 2020.

\bibitem{NaP78}
SV~Nagaev and IF~Pinelis.
\newblock Some inequalities for the distribution of sums of independent random
  variables.
\newblock {\em Theory of Probability \& Its Applications}, 22(2):248--256,
  1978.

\bibitem{PiX97}
Gilles Pisier and Quanhua Xu.
\newblock Non-commutative martingale inequalities.
\newblock {\em Communications in mathematical physics}, 189:667--698, 1997.

\bibitem{PlV13lin}
Yaniv Plan and Roman Vershynin.
\newblock One-bit compressed sensing by linear programming.
\newblock {\em Communications on Pure and Applied Mathematics},
  66(8):1275--1297, 2013.

\bibitem{PlV13}
Yaniv Plan and Roman Vershynin.
\newblock Robust 1-bit compressed sensing and sparse logistic regression: a
  convex programming approach.
\newblock {\em IEEE Transactions on Information Theory}, 59(1):482--494, 2013.

\bibitem{Ran02}
Narcisse Randrianantoanina.
\newblock Non-commutative martingale transforms.
\newblock {\em Journal of Functional Analysis}, 194(1):181--212, 2002.

\bibitem{Rob62}
Lawrence Roberts.
\newblock Picture coding using pseudo-random noise.
\newblock {\em IRE Transactions on Information Theory}, 8(2):145--154, 1962.

\bibitem{Ros70}
H.~Rosenthal.
\newblock On the subspaces of {$L\sp{p}$} {$(p>2)$} spanned by sequences of
  independent random variables.
\newblock {\em Israel J. Math.}, 8:273--303, 1970.

\bibitem{roth2015covariance}
Kilian Roth, Jawad Munir, Amine Mezghani, and Josef~A Nossek.
\newblock Covariance based signal parameter estimation of coarse quantized
  signals.
\newblock In {\em 2015 IEEE International Conference on Digital Signal
  Processing (DSP)}, pages 19--23. IEEE, 2015.

\bibitem{tewksbury1978oversampled}
S~Tewksbury and RW~Hallock.
\newblock Oversampled, linear predictive and noise-shaping coders of order ${N}
  > 1$.
\newblock {\em IEEE Transactions on Circuits and Systems}, 25(7):436--447,
  1978.

\bibitem{tropp2012user}
Joel~A Tropp.
\newblock User-friendly tail bounds for sums of random matrices.
\newblock {\em Foundations of computational mathematics}, 12(4):389--434, 2012.

\bibitem{van1966spectrum}
John~H Van~Vleck and David Middleton.
\newblock The spectrum of clipped noise.
\newblock {\em Proceedings of the IEEE}, 54(1):2--19, 1966.

\bibitem{vershynin2018high}
Roman Vershynin.
\newblock {\em High-dimensional probability: An introduction with applications
  in data science}, volume~47.
\newblock Cambridge University Press, 2018.

\bibitem{xiao2023one}
Yu-Hang Xiao, Lei Huang, David Ram{\'\i}rez, Cheng Qian, and Hing~Cheung So.
\newblock One-bit covariance reconstruction with non-zero thresholds: Algorithm
  and performance analysis.
\newblock {\em arXiv:2303.16455}, 2023.

\bibitem{yang2023plug}
Tianyu Yang, Johannes Maly, Sjoerd Dirksen, and Giuseppe Caire.
\newblock {Plug-in Channel Estimation with Dithered Quantized Signals in
  Spatially Non-Stationary Massive MIMO Systems}.
\newblock {\em arXiv:2301.04641}, 2023.

\end{thebibliography}
\bibliographystyle{plain}

\end{document}